\documentclass[12pt,twoside,reqno]{amsart}

\usepackage{microtype}
\usepackage[OT1]{fontenc}
\usepackage{type1cm}
\usepackage{amssymb}
\usepackage{enumerate}
\usepackage{xcolor}
\usepackage{mathrsfs}
\usepackage{dsfont}
\usepackage{setspace}
\usepackage{tikz}
\usetikzlibrary{arrows}

\usepackage[left=2.3cm,top=3cm,right=2.3cm]{geometry}
\numberwithin{equation}{section}

\theoremstyle{plain}
\newtheorem{theorem}{Theorem}[section]

\newtheorem{lemma}[theorem]{Lemma}

\theoremstyle{remark}

\newtheorem{example}[theorem]{Example}

\theoremstyle{definition}

\renewcommand{\AA}{\mathcal{A}}
\newcommand{\BB}{\mathcal{B}}
\newcommand{\HH}{\mathcal{H}}
\newcommand{\LL}{\mathcal{L}}
\newcommand{\PP}{\mathcal{P}}
\newcommand{\MM}{\mathcal{M}}

\newcommand{\CC}{\mathcal{C}}
\newcommand{\FF}{\mathcal{F}}
\newcommand{\DD}{\mathcal{D}}
\newcommand{\II}{\mathcal{I}}

\newcommand{\UU}{\mathcal{U}}
\newcommand{\R}{\mathbb{R}}

\newcommand{\Z}{\mathbb{Z}}
\newcommand{\N}{\mathbb{N}}
\newcommand{\hhh}{\mathtt{h}}
\newcommand{\iii}{\mathtt{i}}
\newcommand{\jjj}{\mathtt{j}}
\newcommand{\kkk}{\mathtt{k}}

\newcommand{\eps}{\varepsilon}
\newcommand{\fii}{\varphi}

\newcommand{\la}{\langle}
\newcommand{\ra}{\rangle}

\newcommand{\dd}{\,\mathrm{d}}

\DeclareMathOperator{\dist}{dist}
\DeclareMathOperator{\diam}{diam}
\DeclareMathOperator{\proj}{proj}

\DeclareMathOperator{\spt}{spt}

\DeclareMathOperator{\cent}{cent}
\DeclareMathOperator{\Tan}{Tan}
\DeclareMathOperator{\TD}{\mathcal{TD}}

\DeclareMathOperator{\MD}{\mathcal{MD}}
\DeclareMathOperator{\spl}{SPL}

\DeclareMathOperator{\rad}{radius}

\DeclareMathOperator{\essinf}{ess\,inf}

\DeclareMathOperator{\dimloc}{dim_{loc}}

\DeclareMathOperator{\ldimloc}{\underline{dim}_{loc}}

\DeclareMathOperator{\dimh}{dim_H}

\DeclareMathOperator{\ldimh}{\underline{dim}_H}

\DeclareMathOperator{\dimp}{dim_p}

\begin{document}

\title{Dynamics of the scenery flow and conical density theorems}

\author{Antti K\"aenm\"aki}
\address{Department of Mathematics and Statistics \\
         P.O.\ Box 68 (Gustaf H\"allstr\"omin katu 2b) \\
         FI-00014 University of Helsinki \\
         Finland}
\email{antti.kaenmaki@helsinki.fi}

\subjclass[2000]{Primary 28A80; Secondary 37A10, 28A75, 28A33.}
\keywords{Conical densities, rectifiability, scenery flow, tangent measure.}

\begin{abstract}
  Conical density theorems are used in the geometric measure theory to derive geometric information from given metric information. The idea is to examine how a measure is distributed in small balls. Finding conditions that guarantee the measure to be effectively spread out in different directions is a classical question going back to Besicovitch \cite{Besicovitch1938} and Marstrand \cite{Marstrand1954}. Classically, conical density theorems deal with the distribution of the Hausdorff measure.
  
  The process of taking blow-ups of a measure around a point induces a natural dynamical system called the scenery flow. Relying on this dynamics makes it possible to apply ergodic-theoretical methods to understand the statistical behavior of tangent measures. This approach was initiated by Furstenberg \cite{Furstenberg1970, Furstenberg2008} and greatly developed by Hochman \cite{Hochman2010}. The scenery flow is a well-suited tool to address problems concerning conical densities.
  
  In this survey, we demonstrate how to develop the ergodic-theoretical machinery around the scenery flow and use it to study conical density theorems.
\end{abstract}

\maketitle

\begingroup
\singlespacing
\tableofcontents
\endgroup

\section{Introduction}

Rectifiability is one of the most fundamental concepts of geometric measure theory. A rectifiable set in the plane is a set that is smooth in a certain measure-theoretic sense, extending the idea of a rectifiable curve. A set is purely unrectifiable if its intersection with any rectifiable set is negligible. These concepts form a natural pair since every set can be decomposed into rectifiable and purely unrectifiable parts. The foundations of geometric measure theory were laid by Besicovitch \cite{Besicovitch1928, Besicovitch1929, Besicovitch1938}. He introduced the theory of rectifiable sets by describing the structure of the subsets of the plane having finite $\HH^1$-measure. Besicovitch's work was extended to higher dimensions by Federer \cite{Federer1947}. Relationships between densities and rectifiability was extensively studied in \cite{Marstrand1954, Marstrand1961, Marstrand1964, Mattila1975, Moore1950, MorseRandolph1944}. To a great extent, geometric measure theory is about studying rectifiable and purely unrectifiable sets.

Conical density results are used to derive geometric information from metric information. The idea is to study how a measure is distributed in small balls. Upper conical density results related to Hausdorff measure are naturally linked to rectifiability; see \cite{Besicovitch1938, Federer1969, Marstrand1954, Mattila1988, Mattila1995, Salli1985}. Conical density results for more general measures were introduced in \cite{CsornyeiKaenmakiRajalaSuomala2010, FengKaenmakiSuomala2012, KaenmakiRajalaSuomala2016, KaenmakiSuomala2008, KaenmakiSuomala2004, SahlstenShmerkinSuomala2013}. Applications of conical densities have been found in the study of porosities; see \cite{KaenmakiSuomala2008, KaenmakiSuomala2004, Mattila1988}. They have also been applied in the removability questions for Lipschitz harmonic functions; see \cite{Lorent2003, MattilaParamonov1995}.

Taking tangents is a standard tool in analysis. Tangents are usually more regular than the original object and they often capture the local behavior. Understanding how tangents behave at many points gives information about the global structure as well. For example, tangents of a differentiable function are affine maps, and they capture the full behavior of the function. Preiss \cite{Preiss1987} introduced the more general notion of a tangent measure and employed it to characterize rectifiable sets by the existence of densities. Tangent measures are useful because, again, they are more regular than the original measure (for example, tangent measures of rectifiable measures are flat) but one can still pass from information about the tangent measure to the original measure. As another example of the general idea, the tangent sets and measures on self-affine sets have a regular product structure which is absent in the more complicated original object; see \cite{BandtKaenmaki2013, FergusonFraserSahlsten2015, KaenmakiKoivusaloRossi2015, KaenmakiOjalaRossi2016}. The process of taking blow-ups of a measure or a set around a point in fact induces a natural dynamical system consisting in zooming-in around the point. This scenery flow opens a door to ergodic-theoretic methods, which were pioneered by Furstenberg in \cite{Furstenberg1970} and then in more developed form in \cite{Furstenberg2008}, with a comprehensive theory developed by Hochman in \cite{Hochman2010}.

Tangent measures are weak limits of the scenery flow. By applying ergodic theory, we get information about the statistical behavior of tangent measures. Such empirical distributions that appear by magnifying around a typical point are called tangent distributions. This is often the correct class of tangent objects to consider in questions that involve some notion of dimension. Indeed, the sequence along which a single tangent measure arises can be very sparse, and hence, such a measure might not give any information about the original object. 
The idea behind the scenery flow has been examined in many occasions. Authors have considered the scenery flow for specific sets and measures arising from dynamics; see e.g.\ \cite{BedfordFisher1996, BedfordFisher1997, BedfordFisherUrbanski2002, Zahle88}. M\"orters and Preiss \cite{MortersPreiss1998} proved the surprising fact that, when dealing with regular measures, the tangent distributions are Palm distributions, which are distributions with a strong degree of symmetry and translation invariance. Hochman \cite{Hochman2010} then showed that a similar phenomenon holds for all Radon measures: he proved that tangent distributions for any measure are almost everywhere quasi-Palm distributions, which is a weaker notion than Palm but still represents a strong spatial invariance. Hochman named distributions which are scale-invariant and enjoy the quasi-Palm property as fractal distributions.

It turns out that fractal distributions are well-suited to address problems concerning conical densities. The cones in question do not change under magnification and this allows to pass information between the original measure and its tangent distributions. In fact, we will show that, perhaps surprisingly, most of the known conical density results are, in some sense, a manifestation of rectifiability.

\section*{Acknowledgements}
These are the lecture notes of the course I gave in the IMPAN while I was a visiting professor in the Simons Semester ``Dynamical Systems'' in Banach Center during fall 2015. This work was partially supported by the  grant  346300 for IMPAN from the Simons Foundation and the matching 2015-2019 Polish MNiSW fund. I thank the students of the course for making numerous helpful remarks and the IMPAN for an excellent working environment.

\section{Conical densities} \label{sec:conical-densities}

We begin by reviewing a number of conical density results. The purpose of this section is to exhibit how conical densities can be studied by analytical methods.
The presentation here follows \cite{CsornyeiKaenmakiRajalaSuomala2010, Kaenmaki2010, KaenmakiSuomala2008, KaenmakiSuomala2004, Mattila1988, Salli1985}.

\subsection{Preliminaries}
For simplicity, we shall work on $\R^2$ which we equip with the Euclidean norm and the induced metric. The closed ball centered at $x \in \R^2$ with radius $r>0$ is denoted by $B(x,r)$ and the unit circle by $S^1$. For $x \in \R^2$, $\theta \in S^1$, $0<\alpha\le 1$, and $r>0$, we set
\begin{align*}
  X(x,r,\theta,\alpha) &= \{ y \in \R^2 : |\sin \sphericalangle (\theta,y-x)| < \alpha \} \cap B(x,r) \\
  &= \{ y \in \R^2 : |\proj_{\theta^\bot}(y-x)| < \alpha|y-x| \} \cap B(x,r).
\end{align*}
Here $\sphericalangle(\theta,y-x)$ is the opening angle between the vectors $\theta$ and $y-x$, and $\proj_{\theta^\bot}$ is the orthogonal projection onto the orthogonal complement of $\{ t\theta : t \in \R \}$. For notational convenience we identify $\theta \in S^1$ and the line $\{ t\theta : t \in \R \}$ determined by it.

The \emph{$s$-dimensional Hausdorff measure} $\HH^s$ is defined by
\begin{equation*}
  \HH^s(E) = \lim_{\delta \downarrow 0} \HH^s_\delta(E),
\end{equation*}
where
\begin{equation*}
  \HH^s_\delta(E) = \inf\biggl\{ \sum_i \diam(A_i)^s : E \subset \bigcup_i A_i \text{ and } \diam(A_i) \le \delta  \biggr\}.
\end{equation*}
We interpret $0^0 = 1$ and $\diam(\emptyset)^s = 0$ for all $s \in \R$. Without loss of generality, we may assume that the sets $A_i$ used in the definition are closed. It follows that $\HH^s$ is a Borel regular measure. It can be shown that $\HH^1$ is the length measure and $\HH^2$ is a constant multiple of the Lebesgue measure; in fact, $\HH^2(B(x,r)) = (2r)^2$ for all $x \in \R^2$ and $r>0$. Furthermore, if $E \subset \R^2$ and $0 \le s < t < \infty$, then $\HH^s(E) < \infty$ implies $\HH^t(E) = 0$. The \emph{Hausdorff dimension} of a set $E \subset \R^2$ is
\begin{equation*}
  \dimh(E) = \inf\{ s : \HH^s(E) < \infty \}.
\end{equation*}
The Hausdorff dimension of $E$ is thus the unique number for which $\HH^s(E) = \infty$ for all $s < \dimh(E)$ and $\HH^t(E) = 0$ for all $t > \dimh(E)$.

A set $E \subset \R^2$ is called \emph{rectifiable} if there are countably many Lipschitz maps $f_i \colon \R \to \R^2$ so that
\begin{equation*}
  \HH^1\Bigl( E \setminus \bigcup_i f_i(\R) \Bigr) = 0.
\end{equation*}
Observe that $\dimh(E) \le 1$ for all rectifiable sets $E$. A set $E \subset \R^2$ is \emph{purely unrectifiable} if $\HH^1(E \cap M) = 0$ for all rectifiable sets $M \subset \R^2$. Every set $E$ with $\HH^1(E) < \infty$ can be decomposed into rectifiable and purely rectifiable parts; see \cite[Theorem 15.6]{Mattila1995}. Although a rectifiable set $E$ with $\HH^1(E) < \infty$ can be dense (for example, consider the union of a line segment and points with rational coordinates), it has a measure-theoretical tangent at almost every point; see \cite[Theorem 15.19]{Mattila1995}.

\begin{lemma} \label{thm:lemma1}
  Let $E \subset \R^2$. If for every $x \in E$ there are $\theta \in S^1$, $0<\alpha<1$, and $r>0$ so that
  \begin{equation*}
    E \cap X(x,r,\theta,\alpha) = \emptyset,
  \end{equation*}
  then $E$ is rectifiable.
\end{lemma}

\begin{proof}
  Let $\{ \theta_i \}_{i \in \N}$ be dense on $S^1$, $\{ \alpha_j \}_{j \in \N}$ be dense on $(0,1)$, and $\{ r_k \}_{k \in \N}$ be dense on $(0,\infty)$. For each $x \in E$ denote the associated $\theta \in S^1$, $0<\alpha<1$, and $r>0$ given by the assumption of the lemma by $\theta_x$, $\alpha_x$, and $r_x$, respectively. Define
  \begin{equation*}
    E_{h,i,j,k} = \{ x \in E : |\theta_x-\theta_i| \le 1/h, \text{ } \alpha_j \le \alpha_x - 1/h, \text{ and } r_k \le r_x \}
  \end{equation*}
  for all $h,i,j,k \in \N$. Observe that $E = \bigcup_{h,i,j,k \in \N} E_{h,i,j,k}$. Moreover, if $x \in E_{h,i,j,k}$ and $y \in X(x,r_k,\theta_i,\alpha_j)$, then $|\theta_x-\theta_i| \le 1/h$ and $\alpha_j \le \alpha_x - 1/h$ imply that
  \begin{equation*}
    \frac{|\proj_{\theta_x^\bot}(y-x)|}{|y-x|} \le \frac{|\proj_{\theta_i^\bot}(y-x)|}{|y-x|} + 1/h < \alpha_j + 1/h \le \alpha_x.
  \end{equation*}
  Furthermore, since $r_k \le r_x$ we see that $y \in X(x,r_x,\theta_x,\alpha_x)$. Thus, by the assumption, for each $x \in E_{h,i,j,k}$ we have $E \cap X(x,r_k,\theta_i,\alpha_j) = \emptyset$.

  The above reasoning shows that, by expressing $E$ suitably as a countable union, we may assume that $\theta$, $\alpha$, and $r$ do not depend on $x$. Furthermore, since $E$ is a countable union of sets whose diameters are less than $r$ we may assume that $\diam(E)<r$. If $a \in E$ and $b \in \R^2$ is such that $|\proj_{\theta^\bot}(b-a)| < \alpha|b-a|$, then $b \in X(a,r,\theta,\alpha)$. Thus, by the assumption, such a point $b$ cannot be contained in $E$. Therefore,
  \begin{equation*}
    |\proj_{\theta^\bot}(b) - \proj_{\theta^\bot}(a)| \ge \alpha|b-a|
  \end{equation*}
  for all $a,b \in E$. This means that the map $\proj_{\theta^\bot}|_E$ is one-to-one and its inverse $f = (\proj_{\theta^\bot}|_E)^{-1}$ satisfies
  \begin{equation*}
    |f(x)-f(y)| \le \alpha^{-1}|x-y|
  \end{equation*}
  for all $x,y \in \proj_{\theta^\bot}(E)$. Identifying $\theta^\bot$ with $\R$ we see that $\proj_{\theta^\bot}(E) \subset \R$. Let $f = (f_1,f_2)$ and define
  \begin{equation*}
    g_i(x) = \inf\{ f_i(y) + \alpha^{-1}|x-y| : y \in \proj_{\theta^\bot}(E) \}
  \end{equation*}
  for both $i \in \{ 1,2 \}$. It is now easy to see that $g=(g_1,g_2) \colon \R \to \R^2$ satisfies $g|_{\proj_{\theta^\bot}(E)} = f$ and
  \begin{equation*}
    |g(x)-g(y)| \le \sqrt{2}\alpha^{-1}|x-y|
  \end{equation*}
  for all $x,y \in \R$. Since $E = g(\proj_{\theta^\bot}(E))$ we conclude that $E$ is rectifiable.
\end{proof}

It follows from the Besicovitch density theorem (see e.g.\ \cite[Theorem 2.14(1)]{Mattila1995}) that if $E \subset \R^2$, then
\begin{equation*}
  \lim_{r \downarrow 0} \frac{\HH^2(E \cap B(x,r))}{(2r)^2} = 1
\end{equation*}
for $\HH^2$-almost all $x \in E$. Much less can be said for the $s$-dimensional Hausdorff measure when $s<2$.

\begin{theorem} \label{thm:besicovitch}
  If $E \subset \R^2$ satisfies $\HH^s(E) < \infty$, then
  \begin{equation*}
    2^{-s} \le \limsup_{r \downarrow 0} \frac{\HH^s(E \cap B(x,r))}{(2r)^s} \le 1
  \end{equation*}
  for $\HH^s$-almost all $x \in E$.
\end{theorem}

\begin{proof}
  Let us first prove the left-hand side of the inequality. The proof is basically just the definition of the Hausdorff measure. Let
  \begin{equation*}
    A = \{ x \in E : \limsup_{r \downarrow 0} \frac{\HH^s(E \cap B(x,r))}{(2r)^s} < 2^{-s} \}
  \end{equation*}
  and observe that $A = \bigcup_{k \in \N} A_k$, where
  \begin{equation*}
    A_k = \{ x \in A : \HH^s(E \cap B(x,r)) \le (1-\tfrac{1}{k})r^s \text{ for all } 0<r<\tfrac{1}{k} \}
  \end{equation*}
  for all $k \in \N$. It suffices to show that $\HH^s(A_k)=0$ for all $k \in \N$. Fix $k \in \N$ and let $\eps>0$. Let $\{ U_i \}$ be a $\tfrac{1}{k}$-cover of $A_k$ such that $A_k \cap U_i \ne \emptyset$ for all $i$ and
  \begin{equation*}
    \sum_i \diam(U_i)^s \le \HH^s(A_k) + \eps.
  \end{equation*}
  For each $i$ choose $x_i \in A_k \cap U_i$. Since
  \begin{align*}
    \HH^s(A_k) &\le \sum_i \HH^s(A_k \cap U_i) \le \sum_i \HH^s(E \cap B(x_i,\diam(U_i))) \\
    &\le (1-\tfrac{1}{k})\sum_i \diam(U_i)^s \le (1-\tfrac{1}{k})(\HH^s(A_k)+\eps)
  \end{align*}
  we get, by letting $\eps \downarrow 0$, that $\HH^s(A_k) \le (1-\tfrac{1}{k})\HH^s(A_k)$. This is possible only when $\HH^s(A_k)=0$.

  Let us then prove the right-hand side of the inequality. The proof is based on the Vitali's covering theorem. Since the Hausdorff measure is Borel regular we may assume that $E$ is a Borel set. Let $t>1$ and define
  \begin{equation*}
    A = \{ x \in E : \limsup_{r \downarrow 0} \frac{\HH^s(E \cap B(x,r))}{(2r)^s} > t \}
  \end{equation*}
  As above, it suffices to show that $\HH^s(A)=0$. Let $\eps>0$ and choose an open set $U \supset A$ such that $\HH^s(E \cap U) \le \HH^s(U) < \HH^s(A)+\eps$. Let $\delta>0$ and observe that for every $x \in A$ there are arbitrary small numbers $r$ such that $0<r<\delta/2$, $B(x,r) \subset U$, and
  \begin{equation} \label{eq:besi-estimate}
    \HH^s(E \cap B(x,r)) > t(2r)^s.
  \end{equation}
  From these balls, by Vitali's covering theorem (see e.g.\ \cite[Theorem 2.8]{Mattila1995}), we may choose pairwise disjoint balls $B_1,B_2,\ldots$ such that $\HH^s(A \setminus \bigcup_i B_i) = 0$. Since the same holds also for $\HH^s_\delta$, the sub-additivity of $\HH^s_\delta$ and \eqref{eq:besi-estimate} imply that
  \begin{align*}
    \HH^s_\delta(A) &\le \HH^s_\delta\Bigl( A \cap \bigcup_i B_i \Bigr) \le \sum_i \HH^s_\delta(B_i) \le \sum_i \diam(B_i)^s \\
    &< t^{-1}\sum_i \HH^s(E \cap B_i) \le t^{-1}\HH^s(E \cap U) < t^{-1}(\HH^s(A) + \eps).
  \end{align*}
  Therefore, by letting $\delta \downarrow 0$ and $\eps \downarrow 0$, we see that $\HH^s(A) \le t^{-1}\HH^s(A)$. This is possible only when $\HH^s(A)=0$.
\end{proof}

\subsection{Dimension and conical densities}
By Theorem \ref{thm:besicovitch}, there are arbitrary small radii $r$ so that $\HH^s(E \cap B(x,r)) \approx r^s$. The question we are now interested in is that how the set $E$ is distributed in such balls? Here and hereafter $\lesssim$ means that the inequality involves a multiplicative constant which we do not care about. Of course, $a \approx b$ means that $a \lesssim b$ and $b \lesssim a$.

\begin{theorem} \label{thm:salli}
  If $1 < s \le 2$ and $0< \alpha \le 1$, then there is a constant $\eps = \eps(s,\alpha)>0$ satisfying the following: For every $E \subset \R^2$ with $\HH^s(E)<\infty$ and for each $\theta \in S^1$ it holds that
  \begin{equation*}
    \limsup_{r \downarrow 0} \frac{\HH^s(E \cap X(x,r,\theta,\alpha))}{(2r)^s} \ge \eps
  \end{equation*}
  for $\HH^s$-almost all $x \in E$.
\end{theorem}

\begin{proof}
  Let $\eps>0$ and assume that there are $A \subset E$ with $\HH^s(A)>0$ and $\theta \in S^1$ such that
  \begin{equation} \label{eq:salli1}
    \HH^s(A \cap X(x,r,\theta,\alpha)) < \eps r^s
  \end{equation}
  for all $x \in A$ and $0<r<r_0$. According to Theorem \ref{thm:besicovitch}, there exist $x \in A$ and $0<r<r_0/2$ such that
  \begin{equation} \label{eq:salli2}
    \HH^s(A \cap B(x,r)) \gtrsim r^s.
  \end{equation}
  Let $0<\delta <1$. The set $(\{x\} + \theta^\bot) \cap B(x,r)$ can be covered by approximately $\delta^{-1}$ many balls of radius $\delta r$. Thus, by \eqref{eq:salli2}, there exists $y \in (\{x\} + \theta^\bot) \cap B(x,r)$ for which
  \begin{equation} \label{eq:salli3}
    \HH^s(A \cap B(x,r) \cap \proj_{\theta^\bot}^{-1}(B(y,\delta r))) \gtrsim \delta r^s.
  \end{equation}
  Let $z \in A \cap B(x,r) \cap \proj_{\theta^\bot}^{-1}(B(y,\delta r))$ and $t = \delta/\alpha$. Then
  \begin{equation*}
    \proj_{\theta^\bot}^{-1}(B(y,\delta r)) \setminus B(z,tr) \subset X(z,2r,\theta,\alpha)
  \end{equation*}
  and, by \eqref{eq:salli3} and Theorem \ref{thm:besicovitch},
  \begin{align*}
    \HH^s(E \cap X(z,2r,\theta,\alpha)) &\ge \HH^s(A \cap B(x,r) \cap \proj_{\theta^\bot}^{-1}(B(y,\delta r)) \setminus B(z,tr)) \\
    &\gtrsim \delta r^s - t^sr^s = \delta\Bigl( 1-\frac{\delta^{s-1}}{\alpha^s} \Bigr)r^s \gtrsim \delta r^s.
  \end{align*}
  The last inequality holds since we may choose $\delta>0$ in the beginning so small that $1 - \delta^{s-1}/\alpha^s > 0$. Note that we can do this even if $\delta^{s-1}/\alpha^s$ is multiplied by a constant. By \eqref{eq:salli1}, we must have $\eps \gtrsim \delta$.
\end{proof}

In Theorem \ref{thm:salli} we effectively assume that $\dimh(E)>1$; otherwise the statement is empty. This condition guarantees that the set is scattered enough. If $\dimh(E) \le 1$, then $E$ can be rectifiable. In fact, any set of upper Minkowski dimension strictly less than $1$ can be covered by a single Lipschitz curve; see e.g.\ \cite[Lemma 3.1]{BalkaHarangi2014}. The upper Minkowski dimension is always at least the Hausdorff dimension; see e.g.\ \cite[\S 5.3]{Mattila1995} for the definition and basic properties. If $E$ is rectifiable with $\HH^1(E)<\infty$, then the claim of Theorem \ref{thm:salli} does not hold; see \cite[Theorem 15.19]{Mattila1995}.

The following theorem significantly improves Theorem \ref{thm:salli}. It shows that there are arbitrary small scales so that almost all points of $E$ are well surrounded by $E$.

\begin{theorem} \label{thm:mattila}
  If $1<s\le 2$ and $0<\alpha\le 1$, then there is a constant $\eps=\eps(s,\alpha)>0$ satisfying the following: For every $E \subset \R^2$ with $\HH^s(E)<\infty$ it holds that
  \begin{equation*}
    \limsup_{r \downarrow 0} \inf_{\theta \in S^1} \frac{\HH^s(E \cap X(x,r,\theta,\alpha))}{(2r)^s} \ge \eps
  \end{equation*}
  for $\HH^s$-almost all $x \in E$.
\end{theorem}

\begin{proof}
  Since $S^1$ is compact there are $\theta_1,\ldots,\theta_k \in S^1$ such that $S^1 \subset \bigcup_{i=1}^k X(0,1,\theta_i,\alpha/2)$. Observe that for every $\theta \in S^1$ there exists $i \in \{ 1,\ldots,k \}$ such that
  \begin{equation*}
    X(0,1,\theta_i,\alpha/2) \subset X(0,1,\theta,\alpha).
  \end{equation*}
  Let $\eps>0$ and assume that there is $A \subset E$ with $\HH^s(A)>0$ so that for every $x \in A$ and $0<r<r_0$ there exists $i \in \{ 1,\ldots,k \}$ such that
  \begin{equation*}
    \HH^s(A \cap X(x,r,\theta_i,\alpha/2)) < \eps r^s.
  \end{equation*}
  Let
  \begin{equation*}
    A_i = \{ z \in A : \HH^s(A \cap X(z,r,\theta_i,\alpha/2)) < \eps r^s \text{ for some } 0<r<r_0 \}.
  \end{equation*}
  Note that $0<r<r_0$ above can be chosen arbitrary small. Since $A = \bigcup_{i=1}^k A_i$, Theorem \ref{thm:besicovitch} implies that there are $i \in \{ 1,\ldots,k \}$, $x \in A_i$, and $0<r<r_0/2$ such that
  \begin{equation*}
    \HH^s(A_i \cap B(x,r)) \gtrsim k^{-1}r^s.
  \end{equation*}
  Now we can proceed as in the proof of Theorem \ref{thm:salli}.
\end{proof}

The previous theorem can be naturally generalized to packing measures. The \emph{$s$-dimensional packing measure} $\PP^s$ is defined by
\begin{equation*}
  \PP^s(E) = \inf\biggl\{ \sum_i P^s(E_i) : \{ E_i \}_i \text{ is a countable partition of } E \biggr\},
\end{equation*}
where $P^s(E) = \lim_{\delta \downarrow 0} P^s_\delta(E)$ and
\begin{align*}
  P^s_\delta(E) = \sup\biggl\{ \sum_i \diam(B_i)^s : \;&\{ B_i \}_i \text{ is a mutually disjoint countable collection} \\ &\text{of closed balls centered at $E$ with } \diam(B_i) \le \delta  \biggr\}.
\end{align*}
The packing measure is Borel regular. It holds that $\HH^s(E) \le \PP^s(E)$ for all $A \subset \R^2$; see \cite[Theorem 5.12]{Mattila1995}. Furthermore, if $E \subset \R^2$ is compact with $P^s(E) < \infty$, then $\PP^s(E) = P^s(E)$; see \cite[Theorem 2.3]{FengHuaWen1999}. The \emph{packing dimension} of a set $E \subset \R^2$ is $\dimp(E) = \inf\{ s : \PP^s(E) < \infty \}$.

\begin{theorem} \label{thm:packing}
  If $1<s\le 2$ and $0<\alpha\le 1$, then there is a constant $\eps=\eps(s,\alpha)>0$ satisfying the following: For every $E \subset \R^2$ with $\PP^s(E)<\infty$ it holds that
  \begin{equation*}
    \limsup_{r \downarrow 0} \inf_{\theta \in S^1} \frac{\PP^s( E \cap X(x,r,\theta,\alpha) )}{(2r)^s} \ge \eps \limsup_{r \downarrow 0} \frac{\PP^s(E \cap B(x,r))}{(2r)^s}
  \end{equation*}
  for $\PP^s$-almost every $x \in E$.
\end{theorem}

The packing measure satisfies
\begin{equation*}
  \liminf_{r \downarrow 0} \frac{\PP^s(E \cap B(x,r))}{(2r)^s} = 1
\end{equation*}
for $\PP^s$-almost all $x \in E$ whenever $E \subset \R^2$ is such that $\PP^s(E)<\infty$; see e.g.\ \cite[Theorem 6.10]{Mattila1995}). Compare this to Theorem \ref{thm:besicovitch}. The challenge in the proof of Theorem \ref{thm:packing} comes from the fact that the scales obtained from different estimates do not necessarily match. We omit the proof but the interested reader is referred to the survey \cite{Kaenmaki2010}.

What about more general measures? Most measures are so unevenly distributed that there are no gauge functions that could be used to approximate the measure in small balls. It seems natural to examine the ratios
\begin{equation*}
  \frac{\mu(X(x,r,\theta,\alpha))}{\mu(B(x,r))}.
\end{equation*}
For conical density results in this direction, the reader is referred to \cite{CsornyeiKaenmakiRajalaSuomala2010, KaenmakiRajalaSuomala2016, SahlstenShmerkinSuomala2013}. These results declare that if the dimension of the measure is strictly larger than one, then there are arbitrary small scales where the measure is well distributed. In \S \ref{sec:conical-scenery}, we show that, equipped with appropriate dynamical machinery, most of the referred results are consequences of rectifiability and, in particular, Lemma \ref{thm:lemma1}.

\subsection{Unrectifiability and conical densities}
Besides assuming the dimension to be strictly larger than one, another condition to guarantee the measure to be scattered enough is unrectifiability. A Radon measure $\mu$ is \emph{rectifiable} if $\mu \ll \HH^1$ and there exists a rectifiable set $E \subset \R^2$ such that $\mu(\R^2 \setminus E) = 0$. A Radon measure $\mu$ is \emph{purely unrectifiable} if $\mu(E)=0$ for all rectifiable sets $E \subset \R^2$.

\begin{theorem} \label{thm:unrect}
  If $M > 0$ and $0<\alpha\le 1$, then there is a constant $\eps = \eps(M,\alpha)>0$ satisfying the following: For every $\theta \in S^1$ and purely unrectifiable measure $\mu$ on $\R^2$ with
  \begin{equation} \label{eq:kaenmaki1}
    \limsup_{r \downarrow 0} \frac{\mu(B(x,2r))}{\mu(B(x,r))} < M
  \end{equation}
  for $\mu$-almost all $x \in \R^2$ it holds that
  \begin{equation*}
    \limsup_{r \downarrow 0} \frac{\mu(X(x,r,\theta,\alpha))}{\mu(B(x,r))} \ge \eps
  \end{equation*}
  for $\mu$-almost all $x \in \R^2$.
\end{theorem}

\begin{proof}
  Observe that \eqref{eq:kaenmaki1} implies the existence of a constant $M' = M'(M,\alpha)$ for which
  \begin{equation*}
    \limsup_{r \downarrow 0} \frac{\mu(B(x,3r))}{\mu(B(x,\alpha r/20))} < M'
  \end{equation*}
  for $\mu$-almost all $x \in \R^2$. Let us show that the claim holds with $\eps = (4MM')^{-1}$. If this is not the case, then there exist $\theta \in S^1$, purely unrectifiable $\mu$ satisfying \eqref{eq:kaenmaki1}, and a Borel set $A \subset \R^2$ with $\mu(A)>0$ such that
  \begin{equation} \label{eq:kaenmaki2}
    \mu(X(x,2r,\theta,\alpha)) < \eps\mu(B(x,2r))
  \end{equation} 
  for all $x \in A$ and $0<r<r_0$. We may assume that
  \begin{equation} \label{eq:kaenmaki3}
  \begin{split}
    \mu(B(x,2r)) &\le M\mu(B(x,r)), \\
    \mu(B(x,3r)) &\le M'\mu(B(x,\alpha r/20))
  \end{split}
  \end{equation}
  for all $x \in A$ and $0<r<r_0$. By the Besicovitch density theorem (see e.g.\ \cite[Theorem 2.14(1)]{Mattila1995}), we find $x_0 \in A$ and $0<r<r_0$ such that
  \begin{equation} \label{eq:kaenmaki4}
    \mu(A \cap B(x_0,r)) > \mu(B(x_0,r))/2.
  \end{equation}
  
  Define
  \begin{equation*}
    h(x) = \sup\{ |y-x| : y \in A \cap X(x,r,\theta,\alpha/4) \}
  \end{equation*}
  for all $x \in A \cap B(x_0,r)$. Since $\mu$ is purely unrectifiable we see that, by Lemma \ref{thm:lemma1}, $h(x)>0$ for $\mu$-almost all $x \in A \cap B(x_0,r)$. For each $x \in A \cap B(x_0,r)$ with $h(x)>0$ we choose $y_x \in A \cap X(x,r,\theta,\alpha/4)$ such that $|y_x-x| > 3h(x)/4$. Then
  \begin{equation} \label{eq:kaenmaki5}
    A \cap \proj_{\theta^\bot}^{-1}(B(x,\alpha h(x)/4)) \subset X(x,2h(x),\theta,\alpha) \cup X(y_x,2h(x),\theta,\alpha)
  \end{equation} 
  for all $x \in A \cap B(x_0,r)$. By the $5r$-covering theorem (see e.g.\ \cite[Theorem 2.1]{Mattila1995}) there exists a countable collection of pairwise disjoint balls $\{ \proj_{\theta^\bot}(B(x_i,\alpha h(x_i)/20)) \}_i$ so that $5$ times bigger balls cover the set $\proj_{\theta^\bot}(A \cap B(x_0,r))$. Thus, by \eqref{eq:kaenmaki5}, \eqref{eq:kaenmaki2}, and \eqref{eq:kaenmaki3},
  \begin{align*}
    \mu(A \cap B(x_0,r)) &\le \sum_i \mu(A \cap B(x_0,r) \cap \proj_{\theta^\bot}^{-1}(B(x_i,\alpha h(x_i)/4))) \\
    &\le \eps\sum_i \mu(B(x_i,2h(x_i))) + \eps\sum_i \mu(B(y_{x_i},2h(x_i))) \\
    &\le 2\eps\sum_i \mu(B(x_i,3h(x_i)))
    \le 2\eps M'\sum_i \mu(B(x_i,\alpha h(x_i)/20)) \\
    &\le 2\eps M' \mu(B(x_0,2r))
    \le 2\eps M'M \mu(B(x_0,r))
    = \mu(B(x_0,r))/2.
  \end{align*}
  This is a contradiction with \eqref{eq:kaenmaki4}.
\end{proof}

We shall finish this section by analyzing the sharpness of Theorem \ref{thm:unrect}. In Example \ref{ex:example8}, we show that one cannot generalize the result by taking the infimum over all $\theta \in S^1$ before taking the $\limsup$ as in Theorem \ref{thm:mattila}. On the other hand, Example \ref{ex:example9} shows that the result does not hold without the doubling assumption \eqref{eq:kaenmaki1}.

\begin{example} \label{ex:example8}
  There exists purely unrectifiable measure $\mu$ satisfying \eqref{eq:kaenmaki1} such that for each $0<\alpha<1$ we have
  \begin{equation} \label{eq:examplecondition}
    \lim_{r \downarrow 0} \inf_{\theta \in S^1} \frac{\mu(X(x,r,\theta,\alpha))}{\mu(B(x,r))} = 0
  \end{equation}
  for $\mu$-almost all $x \in \R^2$.
\end{example}

\begin{proof}[Construction.]
  The measure $\mu$ will be $\HH^1$ restricted to a purely unrectifiable compact set $A \subset \R^2$. It is evident that $\mu$ is then purely unrectifiable. We construct the set $A$ using a nested sequence $A_0 \supset A_1 \supset \cdots$ of compact sets. The first set $A_0$ is just the unit ball $B(0,1)$. We define a collection of mappings $f_{i,j}$ with $i\in \N$ and $j \in \{1, \dots, 2i^2\}$ as
  \begin{align*}
    f_{i,j}(x,y)=\frac{1}{2i^2}\big( & (\cos(\alpha_i)x+2j-2i^2-1)-(-1)^j\sin(\alpha_i)y,
    (-1)^j\sin(\alpha_i)x+\cos(\alpha_i)y \big),
  \end{align*}
  where $\alpha_i = 1/\sqrt{i}$. Then for each $n \in \N$ we define the set $A_n$ by
  \[
    A_n = \bigcup_{\substack{i\in \{1,\dots,n\}\\ j_i\in\{1,\dots,2i^2\}}} f_{1,j_1}\circ \cdots \circ f_{n,j_n} (A_0).
  \]
  Finally, we set $A=\bigcap_{n=1}^\infty A_n$. So, in the construction, each level $n-1$ ball contains $2n^2$ many disjoint balls along the diagonal making an angle $\alpha_{n-1}$ with the previous diagonal. See Figure \ref{fig1} for an illustration. We refer to the radius of step $n$ construction ball as $R_n$. That is, $R_0 =1$ and $R_n = \frac{R_{n-1}}{2n^2}$ for $n\ge 1$.

  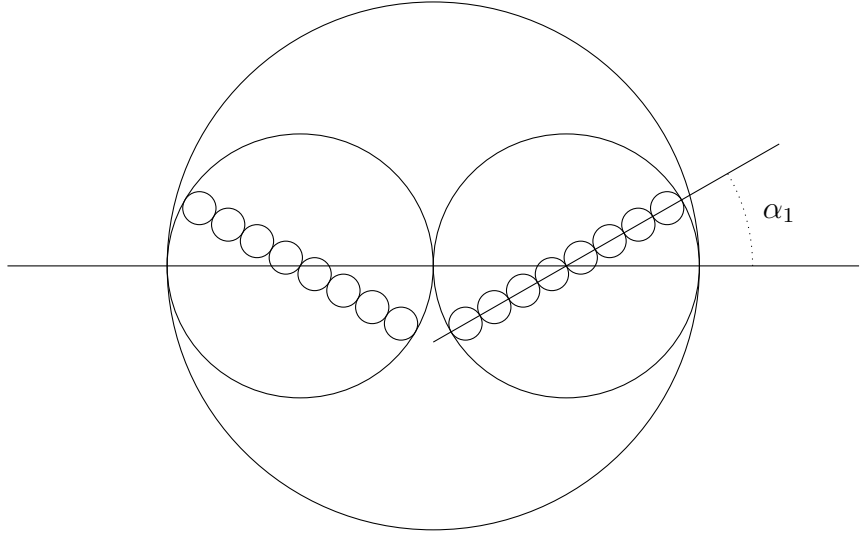
\begin{figure}[t!]
  \begin{center}
    \begin{tikzpicture}[scale=0.7]
      \draw (-8,0) -- (8,0);
      \draw (0,0) circle (5);
      
      \draw (-2.5,0) circle (2.5);
      \draw (2.5,0) circle (2.5);

      \draw (0,-1.443375673) -- (6.5,2.309401077);
      \draw (0.605569431,-1.09375) circle (0.3125);
      \draw (1.146835308,-0.78125) circle (0.3125);
      \draw (1.688101185,-0.46875) circle (0.3125);
      \draw (2.229367062,-0.15625) circle (0.3125);
      \draw (2.770632939,0.15625) circle (0.3125);
      \draw (3.311898816,0.46875) circle (0.3125);
      \draw (3.853164693,0.78125) circle (0.3125);
      \draw (4.394430570,1.09375) circle (0.3125);

      \draw (-0.605569431,-1.09375) circle (0.3125);
      \draw (-1.146835308,-0.78125) circle (0.3125);
      \draw (-1.688101185,-0.46875) circle (0.3125);
      \draw (-2.229367062,-0.15625) circle (0.3125);
      \draw (-2.770632939,0.15625) circle (0.3125);
      \draw (-3.311898816,0.46875) circle (0.3125);
      \draw (-3.853164693,0.78125) circle (0.3125);
      \draw (-4.394430570,1.09375) circle (0.3125);

      \draw [domain=0:30, dotted] plot ({2.5+3.5*cos(\x)}, {3.5*sin(\x)});
      \node[align=center] at (6.5,1) {$\alpha_1$};
    \end{tikzpicture}
    \caption{The picture depicts the first two steps, $A_1$ and $A_2$, in the construction of the set $A$ in Example \ref{ex:example8}.}
    \label{fig1}
  \end{center}
  \end{figure}

  Let us verify that the set $A$ admits the desired properties. Essentially the same argument as in \cite[Theorem III]{Moran1946} shows that $A\subset B(0,1)$ is a compact set with $0<\HH^1(A)\leq 1$: The upper bound is trivial as the sum of the diameters of level $n$ construction balls is always one. If $F \subset B(0,1)$, then there is $n$ and a collection $\BB$ of level $n$ construction balls covering $F \cap A$ so that $\sum_{B \in \BB} \diam(B) < 10\diam(F)$. This gives the lower bound. Moreover, we have $\HH^1(A\cap
  B_n)=R_n\HH^1(A)$ for each construction ball $B_n$ of level $n$. In fact, we have $\HH^1|_A(B(x,r)) \approx r$ for all $x \in \R^2$ and $0<r<1$. This verifies that $\mu$ satisfies \eqref{eq:kaenmaki1}.
  
  For each $x \in A$ there is a unique address $a(x) = \bigl(a_1(x),a_2(x),\ldots\bigr)$ so that  $a_i(x) \in \{1,\dots, 2i^2\}$
  and
  \[
  \{ x \} = \bigcap_{i=1}^\infty f_{1, a_1(x)} \circ \cdots \circ f_{i,a_i(x)} (A_0).
  \]
  By Kolmogorov's zero-one law and the three-series criteria (for
  example, see \cite[Section 17.3]{Loeve1963}), the series $\sum_{i=1}^n(-1)^{a_i(x)}\alpha_i$ diverges for $\HH^1$-almost every $x \in A$. Since $\alpha_i \downarrow 0$ as $i \to \infty$, the sequence $\sum_{i=1}^n (-1)^{a_i(x)}\alpha_i \mod \pi$ has at least two accumulation points for $\HH^1$-almost every $x \in A$. Take such a point $x$ and fix an angle $\beta \in [0, 2\pi]$. It follows that there is $\eps > 0$ so that
  \[
  \limsup_{n \to \infty} \min_{k\in \Z}|\beta - \sum_{i=1}^n(-1)^{a_i(x)}\alpha_i + k\pi| > 4\eps.
  \]
  Let $\theta_\beta$ be the line with an angle $\beta$. We will
  show that
  \begin{equation}\label{eq:notapprtan}
    \limsup_{r \downarrow 0} \frac{\HH^1(A \cap B(x,r) \setminus X(x,r,\theta_\beta,\eps))}{r} > 0.
  \end{equation}
  This means that $\theta_\beta$ is not an approximate tangent of $A$ at
  $x$ and thus $A$ is purely unrectifiable, see for example
  \cite[Corollary 15.20]{Mattila1995}. Take $n \in \N$ large enough so that
  \[
    \min_{k\in \Z}|\beta - \sum_{i=1}^n(-1)^{a_i(x)}\alpha_i+ k\pi| > 2\eps.
  \]
  Since all the $2n^2$ level $n$ construction balls inside the ball
  $f_{1,a_1(x)}\circ \cdots \circ f_{n-1,a_{n-1}(x)}(A_0)$ hit the line
  from $x$ with direction $\sum_{i=1}^n(-1)^{a_i(x)}\alpha_i$, there
  exists $K$ depending only on $\eps$ (it suffices to take
  $K>10/\eps$) so that
  \[
    \#\{m : B_m \cap X(x,R_{n-1},\theta_\beta,\eps) \ne \emptyset\} \le K,
  \]
  where $B_m = f_{1,a_1(x)}\circ \cdots \circ f_{n-1,a_{n-1}(x)} \circ f_{n,m} (A_0)$. This yields an adequate surplus of balls outside the cone $X(x,R_{n-1},\theta_\beta,\eps)$ giving
  \[
    \frac{\HH^1(A\cap B(x,R_{n-1}) \setminus X(x,R_{n-1},\theta_\beta,\eps))}{R_{n-1}} \ge \frac{2n^2-K}{2n^2}\HH^1(A)
  \]
  and therefore \eqref{eq:notapprtan} holds.

  It remains to verify that \eqref{eq:examplecondition} holds. Let $x\in A$
  and $0 < \alpha \le 1$. First observe from the construction that with
  any $n \in \N$ and $y \in A \setminus (f_{1,a_1(x)}\circ \cdots \circ
  f_{n-1,a_{n-1}(x)}(A_0))$ we have
  \[
  \dist(y,x)\ge (1-\cos(\alpha_{n}))R_{n-1} \ge R_{n-1}/(4n) =
  2n^2 R_n/(4n) = n R_n/2.
  \]
  Let $0<r<1$ and choose the $n\in \N$ for which $nR_n \leq
  2r<(n-1)R_{n-1}$. Let $\theta$ be the line perpendicular to the direction
  $\sum_{i=1}^{n-1}(-1)^{a_i(x)}\alpha_i$. Now there are numbers
  $M,n_0\in N$ depending only on $\alpha$ (letting $M>10/\alpha$ and
  $n_0$ so that $\alpha_{n_0-1}<\alpha/10$ will do) so that
  if $n\geq n_0$, then
  \[
  \#\{m : B_m \cap X(x,r,\theta,\alpha) \ne \emptyset\} \le M,
  \]
  where $B_m$'s denote the construction balls of level $n$. Thus
  \[
  \frac{\HH^1(A \cap X(x,r,\theta,\alpha))}{2r} \le \frac{M R_n \HH^1(A)}{nR_n} = \frac{M}{n} \HH^1(A) \longrightarrow 0,
  \]
  as $r\downarrow 0$.
\end{proof}

\begin{example} \label{ex:example9}
  There exists a purely unrectifiable measure $\mu$ and $\theta \in S^1$ such that for each $0<\alpha<1$ we have
  \begin{equation} \label{eq:example9}
    \lim_{r \downarrow 0} \frac{\mu(X(x,r,\theta,\alpha))}{\mu(B(x,r))} = 0
  \end{equation}
  for $\mu$-almost all $x \in \R^2$.
\end{example}

\begin{proof}[Construction.]
  Let us first explain the main idea in the construction. The direction $\theta$ will be the horizontal line. The measure $\mu$ will be constructed by a repetitive process which starts with the Lebesgue measure restricted to a rectangle having sides parallel to the coordinate axis and width strictly smaller than height. All the rectangles in the construction have similar form. Then, at each step, the measure at a given rectangle will be redistributed in such a way that its support is contained in a chain of subrectangles where the location of each subrectangle alternate between left and right side of the original rectangle as one travels vertically. Furthermore, the mass in this process will be divided in such a way that the subrectangles close to the vertical center of the original rectangle get bigger mass than the subrectangles close to the bottom and top.

  We shall then make the following observations: All the vertical curves $\gamma$ miss a lot of the support and we have $\mu(\gamma)=0$. The horizontal curves $\gamma$ can hit only a bounded number of construction pieces and hence $\mu(\gamma)=0$. Since each $C^1$-curve $\gamma$ is a countable union of vertical and horizontal curves we conclude that $\mu(\gamma)=0$ and $\mu$ is purely unrectifiable. Finally, since $\mu$ is concentrated in the vertical centers of the construction rectangles, the cones $X(x,r,\theta,\alpha)$ contain smaller and smaller relative mass. The claim follows.

  Let us now give the detailed construction. We shall construct the measure $\mu$ by using families of maps
  \[
    \{ f_{k,h}^i : k \in \{ 0,\ldots,i-1 \} \text{ and } h \in \{ 0,\ldots,2i^2-1 \} \}_{i=1}^\infty
  \]
  with
  \[
    f_{k,h}^i((x,y)) = \biggl(\frac{(-1)^ki+x}{2i^3},\frac{2ki^2+h+y}{2i^3}\biggr)
  \]
  for every $i \in \{2,3,\dots\}$, $ k \in \{0, \dots, i-1\}$, and $h \in \{0, \dots, 2i^2-1\}$.

  Given $\{f_{k,h}^i\}_{k,h}$, define $F_i$ that maps a measure $\nu$ on $\R^2$ to a measure $F_i(\nu)$ so that for every Borel set $A \subset \R^2$ we get
  \begin{equation}\label{eq:massdistribution}
    F_i(\nu)(A) = \sum_{k = 0}^{i-1} \sum_{h=0}^{2i^2-1} C_i(2i)^{-|h-i^2+\frac{1}{2}|}\nu((f_{k,h}^i)^{-1}(A)),
  \end{equation}
  where the constant $C_i$ is chosen so that $\sum_{k = 0}^{i-1} \sum_{h=0}^{2i^2-1} C_i(2i)^{-|h-i^2+\frac{1}{2}|} = 1$. Applying the map $F_i$ divides the measure into $i$ vertical strips. These strips correspond to the index $k$ in the mappings $f_{k,h}^i$. Inside the strips the measure is divided to $2i^2$ blocks using the index $h$. The measure is concentrated near the vertical centers of the strips by giving different weights to the maps $f_{k,h}^i$ with different values of $h$. See Figure \ref{fig2} for an illustration.

  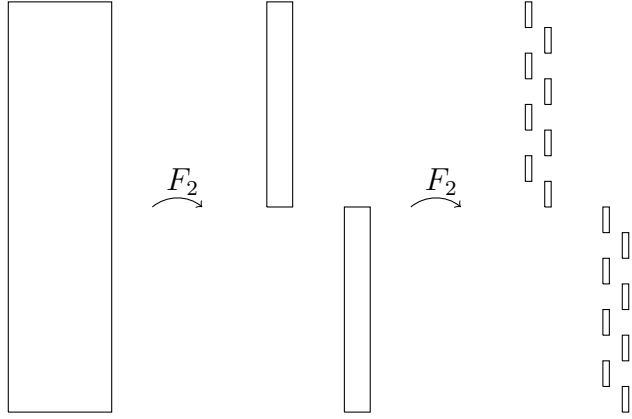
\begin{figure}[t!]
  \begin{center}
    \begin{tikzpicture}[scale=0.17]
      \draw (0,0) -- (0,32) -- (8,32) -- (8,0) -- (0,0);

      \draw[<-] (15,16) arc (50:130:3);
      \node[align=center] at (13.5,18) {$F_2$};
      
      \draw (20,16) -- (20,32) -- (22,32) -- (22,16) -- (20,16);
      \draw (26,0) -- (26,16) -- (28,16) -- (28,0) -- (26,0);

      \draw[<-] (35,16) arc (50:130:3);
      \node[align=center] at (33.5,18) {$F_2$};
      
      \draw (40,30) -- (40,32) -- (40.5,32) -- (40.5,30) -- (40,30);
      \draw (41.5,28) -- (41.5,30) -- (42,30) -- (42,28) -- (41.5,28);
      \draw (40,26) -- (40,28) -- (40.5,28) -- (40.5,26) -- (40,26);
      \draw (41.5,24) -- (41.5,26) -- (42,26) -- (42,24) -- (41.5,24);
      \draw (40,22) -- (40,24) -- (40.5,24) -- (40.5,22) -- (40,22);
      \draw (41.5,20) -- (41.5,22) -- (42,22) -- (42,20) -- (41.5,20);
      \draw (40,18) -- (40,20) -- (40.5,20) -- (40.5,18) -- (40,18);
      \draw (41.5,16) -- (41.5,18) -- (42,18) -- (42,16) -- (41.5,16);

      \draw (46,14) -- (46,16) -- (46.5,16) -- (46.5,14) -- (46,14);
      \draw (47.5,12) -- (47.5,14) -- (48,14) -- (48,12) -- (47.5,12);
      \draw (46,10) -- (46,12) -- (46.5,12) -- (46.5,10) -- (46,10);
      \draw (47.5,8) -- (47.5,10) -- (48,10) -- (48,8) -- (47.5,8);
      \draw (46,6) -- (46,8) -- (46.5,8) -- (46.5,6) -- (46,6);
      \draw (47.5,4) -- (47.5,6) -- (48,6) -- (48,4) -- (47.5,4);
      \draw (46,2) -- (46,4) -- (46.5,4) -- (46.5,2) -- (46,2);
      \draw (47.5,0) -- (47.5,2) -- (48,2) -- (48,0) -- (47.5,0);

    \end{tikzpicture}
    \caption{The picture depicts how the mapping $F_2$ transforms a rectangle in Example \ref{ex:example9}.}
    \label{fig2}
  \end{center}
  \end{figure}

  Let $N_1=0$ and for $i \in \{ 2,3,\ldots \}$ let $N_i$ be the smallest integer so that
  \begin{equation}\label{eq:numberofscales}
    \biggl(1-\frac{C_i}{8(2i)^{i^2-\frac{3}{2}}}\biggr)^{N_i} < \frac{1}{2}.
  \end{equation}
  Integers $N_i$ determine how many times we have to use map $F_i$ when constructing the measure $\mu$ in order to make the resulting measure unrectifiable. With these numbers define a sequence $(I_j)_{j=1}^\infty$ with
  \[
    I_{p+\sum_{i=1}^{t-1}N_i}=t
  \]
  for every $t \in \{2,3,\dots\}$ and $p\in \{1,\dots,N_t\}$. Let also $M_j = \prod_{i=1}^j(2I_i^3)$. Finally, since $F_i$ is a contraction for all $i$ and the space $\PP(X)$ of probability measures on a compact metric space $X$ is compact and metrizable, we may define $\mu$ to be the weak limit of
  \[
    F_{I_1}\circ F_{I_2} \circ \cdots \circ F_{I_m}(\mu_0)
  \]
  as $m \to \infty$. Here $\mu_0$ is any compactly supported Borel probability measure on $\R^2$. Take for example $\HH^1$ restricted to $\{0\}\times[0,1]$. For $i\in \N$, $k\in\{1,\dots, M_{i-1}I_i\}$, and $h\in\{1,\dots,2I_i^2\}$ we define strips
  \[
    S_{i,k} = \spt(\mu) \cap \biggl( \R \times \biggl[\frac{2(k-1)I_i^2}{M_i},\frac{2kI_i^2}{M_i}\biggr] \biggr)
  \]
  and blocks
  \[
    B_{i,k,h} = \spt(\mu) \cap \biggl(\R \times \biggl[\frac{2(k-1)I_i^2+h-1}{M_i},\frac{2(k-1)I_i^2+h}{M_i}\biggr]\biggr).
  \]
  Note that each strip consists of blocks, i.e.\ $S_{i,k} = \bigcup_{h \in \{1,\dots,2I_i^2\}} B_{i,k,h}$ for all $i\in \N$ and $k\in\{1,\dots, M_{i-1}I_i\}$.

  To prove the unrectifiability, we shall consider horizontal and vertical curves separately. We call a $C^1$-curve $\gamma$ \emph{horizontal} (respectively \emph{vertical}) if at every point $\sphericalangle(\gamma',e_1) < 2\pi/5$ (respectively $\sphericalangle(\gamma',e_2) < 2\pi/5$), where $\{e_1,e_2\}$ is the natural basis of $\R^2$. Let us first look at vertical curves: Let $\gamma$ be a $C^1$-curve in $\R^2$ so that $|\frac{\partial\gamma}{\partial y}| \ge\frac{1}{3}|\gamma'|$. Take $i \in \N$. Now for any $k \in \{1, \dots, I_{i+1}-1\}$ and $t \in \{0,\dots,M_i-1\}$  we have either
  \[
    \gamma \cap B_{i+1,2I_{i+1}^3t+k,2I_{i+1}^2} = \emptyset \quad \text{or} \quad \gamma \cap B_{i+1,2I_{i+1}^3t+k+1,1} = \emptyset.
  \]
  This means that when we look at two consecutive strips $S_{i+1,2I_{i+1}^3t+k}$ and $S_{i+1,2I_{i+1}^3t+k+1}$, we see that the curve $\gamma$ cannot meet both the uppermost block of the lower strip and the lowest block of the upper strip. This is because vertically these blocks are next to each other, but horizontally the distance is roughly at least $I_{i+1}$ times the width of the block. Hence the curve $\gamma$ misses more than one fourth of all the end blocks of the strips of the level $I_{i+1}$ construction step. Therefore by iterating and using inequality \eqref{eq:numberofscales}, we get
  \begin{align*}
  \mu(\gamma) & \le \prod_{i=1}^M \biggl(1-\frac{I_{i+1}C_{I_{i+1}}(2I_{i+1})^{-I_{i+1}^2+\frac{1}{2}}}{4} \biggr)\\
  &\le \prod_{m=2}^{I_M-1}\biggl(1-\frac{C_m}{8(2m)^{m^2-\frac{3}{2}}}\biggr)^{N_m} < 2^{-I_M+2} \to 0
  \end{align*}
  as $M \to \infty$.

  Next we look at horizontal curves: Let $\gamma$ be a $C^1$-curve in $\R^2$ so that $|\frac{\partial\gamma}{\partial x}| \ge\frac{1}{3}|\gamma'|$. Take $i \in \N$ and $t \in \{0,\dots,M_i-1\}$. Now there are at most two $k \in \{1,\dots, I_{i+1}\}$ so that
  \[
    \gamma \cap S_{i+1,tI_{i+2}+k} \ne \emptyset.
  \]
  By repeating this observation, we have
  \[
    \mu(\gamma) \le \prod_{i=2}^M\frac{2}{I_i} \to 0
  \]
  as $M \to \infty$. Take any $C^1$-curve $\gamma$ in $\R^2$. Because it can be covered with a countable collection of vertical and horizontal $C^1$-curves defined as above, we have $\mu(\gamma) = 0$. Thus, the measure $\mu$ is purely unrectifiable.

  Finally, let us show that cones around $\theta$ have small measure in the sense of the equality \eqref{eq:example9}. To do this fix $0 < \alpha < 1$ and take the smallest $i_0 \in \{3,4,\dots\}$ so that
  \begin{equation}\label{eq:anglegivingextrablock}
  \frac{1}{I_{i_0}} < \frac{\sqrt{1-\alpha}}{4}.
  \end{equation}
  Now take $i \in \{i_0+1,i_0+2,\dots\}$, a point $x \in \spt(\mu)$, and a radius $r \in [M_i^{-1},M_{i-1}^{-1}]$. Let $k_1 \in \N$ so that $x \in S_{i,k_1}$. Assume that there are at most two  $k' \in \N$ so that
  \[
    X(x,r,\theta,\alpha) \cap S_{i+1,k'} \ne \emptyset.
  \]
  Then
  \begin{equation}\label{eq:coneestimate1}
    \mu(X(x,r,\theta,\alpha)) \le \frac{2\mu(B(x,r))}{I_{i+1}}.
  \end{equation}
  Assume then that there are at least three such $k'$. If this is the case, then the cone $X(x,r,\theta,\alpha)$ must hit another large vertical strip $S_{i,k_2}$ with $k_2 \in \{k_1 - 1, k_1+1\}$.
  Inequality \eqref{eq:anglegivingextrablock} yields the existence of a block $B_{i,k_1,u} \subset B(x,r)$ whose vertical distance to the centre of the strip $S_{i,k_1}$ is strictly less than the vertical distance from the centre of the strip $S_{i,k_2}$ to any of the blocks $B_{i,k_2,u'}$ that intersect the cone $X(x,r,\theta,\alpha)$. Now the fact that we concentrated measure to the centre using equation \eqref{eq:massdistribution} gives
  \[
    \mu(B_{i,k_1,u})\ge \frac{(2I_i)^u}{2\sum_{p=1}^{u-1}(2I_i)^p}\mu(X(x,r,\theta,\alpha))
  \]
  and hence
  \[
    \mu(X(x,r,\theta,\alpha)) \le \frac{2\mu(B(x,r))}{I_i}.
  \]
  This together with \eqref{eq:coneestimate1} shows \eqref{eq:example9} as $i$ tends to infinity.
\end{proof}

\section{Scenery flow} \label{sec:scenery-flow}

In this section, we build the ergodic-theoretic machinery around the scenery flow. The presentation follows \cite{Hochman2010, KaenmakiRossi2016, KaenmakiSahlstenShmerkin2015a}.

\subsection{Measure preserving systems}
Let $(X,\BB,P)$ be a probability space. Here $X$ is a metric space, $\BB$ is the Borel $\sigma$-algebra on $X$, and $P \in \PP(X)$. We say that a Borel measurable transformation $T \colon X \to X$ \emph{preserves} $P$ if $TP=P$, where $TP$ is the push-forward of $P$ under $T$. A \emph{measure-preserving system} is a tuple $(X,\BB,P,T)$ where $(X,\BB,P)$ is a probability space and $T$ is an action of a semigroup under composition by transformations that preserve $P$. In other words, there is a semigroup $S$ and for each $s \in S$ there is a map $T_s \colon X \to X$ that preserves $P$ such that $T_{s+s'} = T_s \circ T_{s'}$. The underlying semigroup will always be one of $\N$, $\Z$ (in which case we speak of measure-preserving maps, since the action is determined by $T_1$), or $\R_+$, $\R$ (in which case we speak of \emph{flows}). Measures which are preserved by all transformations in the semigroup are called \emph{invariant}. Since $\BB$ is the Borel $\sigma$-algebra, we will, for notational convenience, make no explicit reference to it.

Let $(X,P,T)$ be measure preserving flow. We say that a set $A \in \BB$ is \emph{invariant} if $P(T_s^{-1}(A) \triangle A) = 0$ for all $s$. The flow is \emph{ergodic} if any invariant set $A \in \BB$ has either zero or full $P$-measure. For a given action $T$ on $X$, ergodic measures are the extremal points of the convex set of all probability measure which are preserved by $T$. The Birkhoff ergodic theorem asserts that if $f \in L^1(\mu)$ and the flow is ergodic, then
\begin{equation*}
  \lim_{T \to \infty} \frac{1}{T} \int_0^T f \circ T_s(x) \dd s = \int_X f \dd\mu
\end{equation*}
for $\mu$-almost all $x \in X$. For discrete actions, the same holds replacing the left-hand side by $\lim_{n \to \infty} \tfrac{1}{n} \sum_{k=0}^{n-1} f \circ T^k(x)$.

A measure preserving system $(X,P,T)$ can be decomposed into ergodic parts according to the ergodic decomposition theorem; for example, see \cite[Chapter II, Theorems 6.1 and 6.4]{Mane1987}. It says that there exists a Borel map $\omega \mapsto P_\omega$ from $X$ to $\PP(X)$ such that for $P$-almost every $\omega$ it holds that each $(X,P,T)$ is measure preserving and ergodic, and
\begin{equation*}
  P = \int_X P_\omega \dd P(\omega).
\end{equation*}
Moreover, this map is unique up to sets of zero $P$-measure. The measures $P_\omega$ are called the \emph{ergodic components} of $P$.

\subsection{Tangent distributions}
Equip $\R^2$ with the norm $\|(x,y)\| = \max\{|x|,|y|\}$ and the induced metric. The closed unit ball is now $[-1,1]^2$. Let $\MM = \MM(\R^2)$ denote the space of Radon measures on $\R^2$ equipped with the weak topology. Let $\PP(X)$ be the space of probability measures on $X$. When $X$ is a compact metric space, then $\PP(X)$ equipped with the weak topology is compact and metrizable. Probability measures on measures will be called distributions. Writing $\mu \sim \nu$ we indicate that the measures $\mu$ and $\nu$ are \emph{equivalent}, which means that $\mu$ and $\nu$ have the same null sets. To indicate that $x$ is chosen randomly according to $\mu$ we write $x \sim \mu$.

If $\mu \in \MM$ and $\mu(A)>0$, then $\mu|_A \in \MM$ is the restriction of $\mu$ on $A$. Furthermore, if $0<\mu(A)<\infty$, then the normalized version of $\mu|_A$ is $\mu_A \in \MM$ defined by
\begin{equation*}
  \mu_A(B) = \frac{\mu|_A(B)}{\mu(A)}.
\end{equation*}
We also write
\begin{equation*}
  \mu^* = \frac{\mu}{\mu([-1,1]^2)} \quad \text{and} \quad
  \mu^\square = (\mu|_{[-1,1]^2})^* = \mu_{[-1,1]^2} = \frac{\mu|_{[-1,1]^2}}{\mu([-1,1]^2)}.
\end{equation*}
The operations $*$ and $\square$ are called \emph{normalization operations} and we apply them pointwise to sets and distributions. For example, if $\AA \subset \MM$, then $\AA^* = \{ \mu^* : \mu \in \AA \}$. In particular,
\begin{equation*}
  \MM^* = \{ \mu \in \MM : \mu([-1,1]^2)=1 \} \quad \text{and} \quad
  \MM^\square = \PP([-1,1]^2).
\end{equation*}
In the same way, $P^*$ is the push-forward of the distribution $P$ under the map $\mu \mapsto \mu^*$.

We define $T_x \colon \R^2 \to \R^2$, $T_x(y) = y-x$, and $S_t \colon \R^2 \to \R^2$, $S_t(y) = e^ty$, for all $x \in \R^2$ and $t \in \R_+$.
Thus,
\begin{equation*}
  T_x\mu(A) = \mu(A+x) \quad \text{and} \quad
  S_t\mu(A) = \mu(e^{-t}A).
\end{equation*}
Denoting
\begin{equation*}
  T_x^\square \mu = (T_x\mu)^\square \quad \text{and} \quad
  S_t^\square \mu = (S_t\mu)^\square,
\end{equation*}
and similarly for $*$, we see that e.g.\
\begin{equation*}
  S_t^\square \mu(A) = \frac{\mu(e^{-t}A)}{\mu([-e^{-t},e^{-t}]^2)}
\end{equation*}
for all $A \subset [-1,1]^2$. Thus, we get $S_t^\square\mu$ from $\mu$ by scaling $\mu|_{[-e^{-t},e^{-t}]^2}$ into $[-1,1]^2$ and normalizing. By the exponential scaling, $(S_t^\square)_{t \in \R_+}$ is a one-sided flow on $\{ \mu \in \MM^\square : 0 \in \spt(\mu) \}$ and it is called a \emph{scenery flow} at $0$. We note that $S_t^\square$ is discontinuous (at measures $\mu$ with $\mu(\partial[-e^{-t},e^{-t}]^2)>0$) and $\{ \mu \in \MM^\square : 0 \in \spt(\mu) \}$ is not closed. Let $\mu \in \MM$ and $x \in \spt(\mu)$. Defining
\begin{equation*}
  \mu_{x,t}^\square = S_t^\square(T_x\mu),
\end{equation*}
the one-sided flow $(\mu_{x,t}^\square)_{t \in \R_+}$ is called the \emph{scenery flow} at $x$. Weak limits of $\mu_{x,t}^\square$ are the \emph{tangent measures} of $\mu$ at $x$. Note that the tangent measures defined by Preiss \cite{Preiss1987} correspond to taking weak limits of $\mu_{x,t}^* = S_t^*(T_x\mu)$. In the original definition, the tangent measures were constructed without restricting the measure. The family of tangent measures of $\mu$ at $x$ is denoted by $\Tan(\mu,x)$.

A \emph{tangent distribution} of $\mu$ at $x \in \spt(\mu)$ is any weak limit of
\begin{equation*}
  \la \mu \ra_{x,T}^\square = \frac{1}{T} \int_0^T \delta_{\mu_{x,t}^\square} \dd t
\end{equation*}
as $T \to \infty$. Here the integration makes sense since we are on a convex subset of a topological linear space. The family of tangent distributions of $\mu$ at $x$ is denoted by $\TD(\mu,x)$. Since $\TD(\mu,x)$ is defined as a set of accumulation points in a compact space $\PP(\MM^\square)$ the subspace $\TD(\mu,x)$ is always non-empty and compact at $x \in \spt(\mu)$. If the limit is unique, then the collection of views will have well defined statistics when zooming in. It is straightforward to see that the support of each $P \in \TD(\mu,x)$ is in $\Tan(\mu,x)$. Observe that, by the discontinuity of $S_t^\square$, it is not a priori clear that $(\MM^\square,P,S^\square)$ is a measure preserving system for $P \in \TD(\mu,x)$.

\subsection{Fractal distributions}
If $U \subset \R^2$ and $\mu \in \MM$, then we define the \emph{$U$-diffusion} of $\mu$ by
\begin{equation*}
  \la \mu \ra_U = \int_U \delta_{T_x\mu} \dd \mu(x).
\end{equation*}
In other words, $\la\mu\ra_U$ is the distribution of the random measure $\nu$ obtained by choosing $x \in U$ according to $\mu$ and setting $\nu = T_x\mu$. We say that a distribution $P$ on $\MM^*$ is \emph{quasi-Palm} if
\begin{equation*}
  P \sim \int_{\MM^*} \la\mu\ra_U^* \dd P(\mu)
\end{equation*}
for every open and bounded neighborhood $U$ of the origin. Here $\la\mu\ra_U^* = \int_U \delta_{T^*_x\mu} \dd\mu(x)$. This means that $P$ is quasi-Palm if a Borel set $\AA$ satisfies $P(\AA)=1$ if and only if $T^*_z\nu \in \AA$ for $P$-almost all $\nu \in \MM^*$ and $\nu$-almost all $z \in \R^2$.

Recall that a distribution $P$ on $\MM^*$ is $S^*$-invariant if $S^*_tP=P$ for all $t$. If $P$ is an $S^*$-invariant distribution, then $P^\square$ is an $S^\square$-invariant distribution on $\MM^\square$ called the \emph{restricted version} of $P$. Similarly, $P$ is called the \emph{extended version} of $P^\square$. This is a one-to-one correspondence between $S^*$- and $S^\square$-invariant distributions.

We say that a distribution $P$ on $\MM^*$ is \emph{fractal distribution} if it is $S^*$-invariant and quasi-Palm. When $P$ is a fractal distribution we shall often refer to its restricted version $P^\square$ as a fractal distribution. This is justified by the one-to-one correspondence between the restricted and extended versions.

The following theorem shows that tangent distributions are fractal distributions.

\begin{theorem} \label{thm:theorem11}
  If $\mu \in \MM$, then $\TD(\mu,x)$ is a collection of restricted fractal distributions for $\mu$-almost all $x \in \R^2$.
\end{theorem}

The result is very surprising. While the $S^\square$-invariance is somewhat expected, there is no a priori reason why the support of a tangent distribution should be in $\{ \mu \in \MM^\square : 0 \in \spt(\mu) \}$ nor why should it satisfy the quasi-Palm property.

We shall next start building machinery to prove the theorem. The proof will be given in \S \ref{sec:MH-proof}. Basically all the results on fractal distributions are based on the interplay with CP-distributions which are Markov processes on the dyadic scaling sceneries of a measure introduced by Furstenberg. For simplicity, let us first consider the situation symbolically.

\subsection{Shift space}
Let $N \ge 2$ and $\Sigma = \{ 1,\ldots,N \}^\N$ be the collection of all infinite words constructed from integers $\{ 1,\ldots,N \}$. We denote the left shift operator by $\sigma$ and equip $\Sigma$ with the usual ultrametric in which the distance between two different words is $2^{-n}$, where $n$ is the first place at which the words differ. The \emph{shift space} $\Sigma$ is clearly compact. If $\iii = i_1i_2\cdots \in \Sigma$, then $\sigma(i_1i_2\cdots) = i_2i_3\cdots$ and $\iii|_n = i_1 \cdots i_n$ for all $n \in \N$. The empty word $\iii|_0$ is denoted by $\varnothing$. We set $\Sigma_n = \{ \iii|_n : \iii \in \Sigma \}$ for all $n \in \N$ and $\Sigma_* = \bigcup_{n=0}^\infty \Sigma_n$. Thus $\Sigma_*$ is the free monoid on $\{ 1,\ldots,N \}$. The concatenation of two words $\iii \in \Sigma_*$ and $\jjj \in \Sigma_* \cup \Sigma$ is denoted by $\iii\jjj$.

The length of $\iii \in \Sigma_* \cup \Sigma$ is denoted by $|\iii|$. For $\iii
\in \Sigma_*$ we set $\iii^- = \iii|_{|\iii|-1}$ and $[\iii] = \{ \iii\jjj \in
\Sigma : \jjj \in \Sigma \}$. The set $[\iii]$ is called a \emph{cylinder set}.
Cylinder sets are open and closed and they generate the Borel $\sigma$-algebra.
If $\iii,\jjj \in \Sigma_*$ such that $[\iii] \cap [\jjj] = \emptyset$, then we
write $\iii \bot \jjj$. The longest common prefix of $\iii,\jjj \in \Sigma_*
\cup \Sigma$ is denoted by $\iii \wedge \jjj$. Thus $\iii = (\iii \wedge
\jjj)\iii'$ and $\jjj = (\iii \wedge \jjj)\jjj'$ for some $\iii',\jjj' \in
\Sigma_* \cup \Sigma$.

\subsection{Adapted distributions}
Let $\PP(X)$ be the set of all Borel probability measures defined on a given metric space $X$. Recall that if $X$ is compact, then $\PP(X)$ is metrizable and compact in the weak topology. We define
\begin{equation} \label{eq:omegagamma}
  \Omega = \{ (\mu,\iii) \in \PP(\Sigma) \times \Sigma : \iii \in \spt(\mu) \}.
\end{equation}
The space $\Omega$ has the subspace topology inherited from the product space $\PP(\Sigma) \times \Sigma$. To simplify notation, we abbreviate $\PP(\Sigma) \times \Sigma$ as $\PP_\Sigma$. Since $(\mu,\iii) \mapsto \mu([\iii|_n])$ is continuous we see that the set $\{ (\mu,\iii) \in \PP_\Sigma : \mu([\iii|_n]) > 0 \}$ is open for all $n \in \N$ and therefore
\begin{equation*}
  \Omega = \bigcap_{n=0}^\infty \{ (\mu,\iii) \in \PP_\Sigma : \mu([\iii|_n]) > 0 \}
\end{equation*}
is a Borel set.

Measures on $\PP(\Sigma)$ and $\PP_\Sigma$ will be called distributions.
We say that a distribution $Q$ on $\PP_\Sigma$ is \emph{adapted} if there exists a distribution $\overline{Q}$ on $\PP(\Sigma)$ such that
\begin{equation} \label{eq:adapted}
  \int_{\PP_\Sigma} f(\mu,\iii) \dd Q(\mu,\iii) = \int_{\PP(\Sigma)} \int_{\Sigma} f(\mu,\iii) \dd\mu(\iii) \dd\overline{Q}(\mu)
\end{equation}
for all $f \in C(\PP_\Sigma)$. Here $C(X)$ is the space of all continuous functions $X \to \R$. In other words, $Q$ is adapted if choosing a pair $(\mu,\iii)$ according to $Q$ can be done in two-step process, by first choosing $\mu$ according to $\overline{Q}$ and then choosing $\iii$ according to $\mu$.

\begin{lemma} \label{thm:adapted_def}
  A distribution $Q$ on $\PP_\Sigma$ is adapted if and only if there exists a distribution $\overline{Q}$ on $\PP(\Sigma)$ such that
  \begin{equation*}
    \int_{\PP_\Sigma} f(\mu,\iii) \dd Q(\mu,\iii) = \int_{\PP(\Sigma)} \int_{\Sigma} f(\mu,\iii) \dd\mu(\iii) \dd\overline{Q}(\mu)
  \end{equation*}
  for all essentially bounded and measurable functions $f \colon \PP_\Sigma \to \R$.
\end{lemma}

\begin{proof}
  Since the other direction is trivial let us assume that $Q$ is adapted and $\overline{Q}$ is as in \eqref{eq:adapted}. As a first step, we will show that the claim holds for the indicator function of an open set. The proof of the claim then follows from Lusin's theorem.
  
  Let $U \subset \PP_\Sigma$ be an open set and choose an increasing sequence $K_1 \subset K_2 \subset \cdots$ of closed sets such that $\bigcup_{n=1}^\infty K_n = U$. For each $n \in \N$, by Urysohn's lemma, there exists a continuous function $f_n$ defined on $\PP_\Sigma$ such that $f_n(\mu,\iii) = 1$ for all $(\mu,\iii) \in K_n$ and $f_n(\mu,\iii) = 0$ for all $(\mu,\iii) \in \PP_\Sigma \setminus U$. Now, by applying the dominated convergence theorem and \eqref{eq:adapted}, we get
  \begin{equation} \label{eq:adapted3}
  \begin{split}
    Q(U) &= \lim_{n \to \infty} \int_{\PP_\Sigma} f_n(\mu,\iii) \dd Q(\mu,\iii) = \lim_{n \to \infty} \int_{\PP(\Sigma)} \int_{\Sigma} f_n(\mu,\iii) \dd\mu(\iii) \dd\overline{Q}(\mu) \\
    &= \int_{\PP(\Sigma)} \int_{\Sigma} \chi_U(\mu,\iii) \dd\mu(\iii) \dd\overline{Q}(\mu).
  \end{split}
  \end{equation}
  Fix $\eps>0$ and let $M \in \N$ be such that $|f(\mu,\iii)| \le M$ for $Q$-almost all $(\mu,\iii) \in \PP_\Sigma$. By Lusin's theorem, there exists a compact set $K \subset \PP_\Sigma$ with $Q(K) \ge 1-(4M)^{-1}\eps$ so that $f|_K$ is continuous. Now, by the Tietze extension theorem, there exists a continuous function $g$ defined on $\PP_\Sigma$ such that $g|_K = f|_K$ and $|g(\mu,\iii)| \le M$ for all $(\mu,\iii) \in \PP_\Sigma$. Applying \eqref{eq:adapted} and \eqref{eq:adapted3}, we get
  \begin{align*}
    \Bigl| \int_{\PP_\Sigma} f(\mu,\iii) &\dd Q(\mu,\iii) - \int_{\PP(\Sigma)}\int_\Sigma f(\mu,\iii) \dd\mu(\iii)\dd\overline{Q}(\mu) \Bigr| \\
    &\le \Bigl| \int_{\PP_\Sigma} f(\mu,\iii) \dd Q(\mu,\iii) - \int_{\PP_\Sigma} g(\mu,\iii) \dd Q(\mu,\iii) \Bigr| \\
    &\qquad\quad+ \Bigl| \int_{\PP(\Sigma)}\int_\Sigma g(\mu,\iii) \dd\mu(\iii)\dd\overline{Q}(\mu) - \int_{\PP(\Sigma)}\int_\Sigma f(\mu,\iii) \dd\mu(\iii)\dd\overline{Q}(\mu) \Bigr| \\
    &\le \int_{\PP_\Sigma} |f(\mu,\iii)-g(\mu,\iii)| \dd Q(\mu,\iii) + \int_{\PP(\Sigma)} \int_\Sigma |g(\mu,\iii)-f(\mu,\iii)| \dd\mu(\iii)\dd\overline{Q}(\mu) \\
    &\le 2MQ(\PP_\Sigma \setminus K) + 2M\int_{\PP(\Sigma)}\int_\Sigma \chi_{\PP_\Sigma \setminus K}(\mu,\iii) \dd\mu(\iii)\dd\overline{Q}(\mu) \\
    &= 4MQ(\PP_\Sigma \setminus K) \le \eps.
  \end{align*}
  This is what we wanted to show.
\end{proof}

\begin{lemma} \label{thm:adapted_omega}
  If $Q$ is an adapted distribution on $\PP_\Sigma$, then $Q(\Omega) = 1$.
\end{lemma}

\begin{proof}
  Since
  \begin{equation*}
    \PP_\Sigma \setminus \Omega = \bigcup_{\mu \in \PP(\Sigma)} \{ \mu \} \times (\Gamma \setminus \spt(\mu)) = \{ (\mu,\iii) \in \PP_\Sigma : \iii \notin \spt(\mu) \}
  \end{equation*}
  Lemma \ref{thm:adapted_def} gives
  \begin{equation*}
    \int_{\PP_\Sigma} \chi_{\PP_\Sigma \setminus \Omega}(\mu,\iii) \dd Q(\mu,\iii) = \int_{\PP(\Sigma)} \int_\Sigma \chi_{\Sigma \setminus \spt(\mu)}(\iii) \dd\mu(\iii)\dd\overline{Q}(\mu) = 0
  \end{equation*}
  as claimed.
\end{proof}

\begin{lemma} \label{thm:adapted2}
  A distribution $Q$ on $\PP_\Sigma$ is adapted if and only if there exists a distribution $\overline{Q}$ on $\PP(\Sigma)$ such that
  \begin{equation} \label{eq:adapted2}
    \int_{\PP_\Sigma} f(\mu)g(\iii) \dd Q(\mu,\iii) = \int_{\PP(\Sigma)} f(\mu) \int_\Sigma g(\iii) \dd\mu(\iii) \dd\overline{Q}(\mu)
  \end{equation}
  for all $f \in C(\PP(\Sigma))$ and $g \in C(\Sigma)$.
\end{lemma}

\begin{proof}
  Since adaptedness clearly implies \eqref{eq:adapted2} let us assume that there exists a distribution $\overline{Q}$ on $\PP(\Sigma)$ such that \eqref{eq:adapted2} holds for all functions in $C(\PP(\Sigma))$ and $C(\Sigma)$. Fix $f \in C(\PP_\Sigma)$ and for each $n \in \N$ define a function $f_n \colon \PP_\Sigma \to \R$ by setting
  \begin{equation*}
    f_n(\mu,\iii) = \sum_{\jjj \in \Sigma_n} \min\{ f(\mu,\hhh) : \hhh \in [\jjj] \} \cdot \chi_{[\jjj]}(\iii)
  \end{equation*}
  for all $(\mu,\iii) \in \PP_\Sigma$. Observe that $(f_n)_{n \in \N}$ is an increasing sequence of continuous functions such that $f_n \to f$ as $n \to \infty$. Moreover, each $f_n$ is defined to be a sum of products, where each product is between functions in $C(\PP(\Sigma))$ and $C(\Sigma)$. Therefore, by the monotone convergence theorem and \eqref{eq:adapted2}, we have
  \begin{align*}
    \int_{\PP_\Sigma} f(\mu,\iii) \dd Q(\mu,\iii) &= \lim_{n \to \infty} \sum_{\jjj \in \Sigma_n} \int_{\PP_\Sigma} \min\{ f(\mu,\hhh) : \hhh \in [\jjj] \} \cdot \chi_{[\jjj]}(\iii) \dd Q(\mu,\iii) \\
    &= \lim_{n \to \infty} \sum_{\jjj \in \Sigma_n} \int_{\PP(\Sigma)} \int_{\Sigma} \min\{ f(\mu,\hhh) : \hhh \in [\jjj] \} \cdot \chi_{[\jjj]}(\iii) \dd\mu(\iii) \dd\overline{Q}(\mu) \\
    &= \int_{\PP(\Sigma)} \int_\Sigma f(\mu,\iii) \dd(\iii) \dd\overline{Q}(\mu).
  \end{align*}
  This is what we wanted to show.
\end{proof}

\begin{lemma} \label{thm:lemma15}
  The family of adapted distributions on $\PP_\Sigma$ is convex and compact.
\end{lemma}

\begin{proof}
  To show the convexity, assume that $Q_1$ and $Q_2$ are adapted distributions on $\PP_\Sigma$. Fix $0<\lambda<1$ and let $P = \lambda Q_1 + (1-\lambda) Q_2$. Since
  \begin{align*}
    \int_{\PP_\Sigma} f(\mu,\iii) \dd P(\mu,\iii) &= \lambda \int_{\PP_\Sigma} f(\mu,\iii) \dd Q_1(\mu,\iii) + (1-\lambda)\int_{\PP_\Sigma} f(\mu,\iii) \dd Q_2(\mu,\iii) \\
    &= \lambda \int_{\PP(\Sigma)} \int_{\Sigma} f(\mu,\iii) \dd\mu(\iii) \dd\overline{Q}_1(\mu) \\
    &\qquad\quad + (1-\lambda) \int_{\PP(\Sigma)} \int_{\Sigma} f(\mu,\iii) \dd\mu(\iii) \dd\overline{Q}_2(\mu) \\
    &= \int_{\PP(\Sigma)} \int_{\Sigma} f(\mu,\iii) \dd\mu(\iii) \dd(\lambda\overline{Q}_1 + (1-\lambda)\overline{Q}_2)(\mu)
  \end{align*}
  for all $f \in C(\PP_\Sigma)$ we see that also $P$ is an adapted distribution on $\PP_\Sigma$.
  
  To show the compactness, let $(Q_n)_{n \in \N}$ be a sequence of adapted distributions on $\PP_\Sigma$. For each $Q_n$ let $\overline{Q}_n$ be a distribution on $\PP(\Sigma)$ as in Lemma \ref{thm:adapted2}. Since $\PP_\Sigma$ and $\PP(\Sigma)$ are compact we may, by going into a subsequence if necessary, assume that there exist distributions $Q_0$ on $\PP_\Sigma$ and $\overline{Q}_0$ on $\PP(\Sigma)$ such that $Q_n \to Q_0$ and $\overline{Q}_n \to \overline{Q}_0$ as $n \to \infty$. According to Lemma \ref{thm:adapted2},
  \begin{equation*}
    \int_{\PP_\Sigma} f(\mu)g(\iii) \dd Q_n(\mu,\iii) = \int_{\PP(\Sigma)} f(\mu) \int_\Sigma g(\iii) \dd\mu(\iii) \dd\overline{Q}_n(\mu)
  \end{equation*}
  for all $f \in C(\PP(\Sigma))$, $g \in C(\Sigma)$, and $n \in \N$. Since the function $\mu \mapsto \int_\Sigma g(\iii) \dd\mu(\iii)$ defined on $\PP(\Sigma)$ is continuous we see that, by letting $n \to \infty$, also $Q_0$ is adapted.
\end{proof}

\subsection{Symbolic CP-distributions}
If $\mu \in \PP(\Sigma)$, then for each $\iii \in \Sigma_*$ we define a measure $\mu_\iii \in \PP(\Sigma)$ by setting
\begin{equation*}
  \mu_\iii([\jjj]) =
  \begin{cases}
    \mu([\iii\jjj])/\mu([\iii]), &\text{if } \mu([\iii])>0, \\
    0, &\text{otherwise},
  \end{cases}
\end{equation*}
for all $\jjj \in \Sigma_*$. We define a mapping $M \colon \PP_\Sigma \to \PP_\Sigma$ by setting $M(\mu,\iii) = (\mu_{\iii|_1},\sigma(\iii))$ for all $(\mu,\iii) \in \PP_\Sigma$. Observe that if $\mu([\iii\jjj])>0$, then
\begin{equation*}
  (\mu_\iii)_\jjj([\hhh]) = \frac{\mu_\iii([\jjj\hhh])}{\mu_\iii([\jjj])} = \frac{\mu([\iii\jjj\hhh])}{\mu([\iii\jjj])} = \mu_{\iii\jjj}([\hhh]).
\end{equation*}
Thus $M^k(\mu,\iii) = (\mu_{\iii|_k},\sigma^k(\iii))$ for all $(\mu,\iii) \in \PP_\Sigma$ and $k \in \N$. Furthermore, if $\mu([\iii|_n]) > 0$ for all $n \in \N$, then also $\mu_{\iii|_1}([\sigma(\iii)|_n]) > 0$ for all $n \in \N$. Hence $M(\Omega) \subset \Omega$ where $\Omega$ is defined in \eqref{eq:omegagamma}. It follows that $M$ restricted to $\Omega$ is continuous.

\begin{lemma} \label{thm:invariant_adapted}
  If $Q$ is an adapted distribution on $\PP_\Sigma$, then $MQ$ is an adapted distribution on $\PP_\Sigma$.
\end{lemma}

\begin{proof}
  Fix $f \in C(\PP(\Sigma))$ and $g \in C(\Sigma)$. For each $n \in \N$ define a function $g_n \colon \Sigma \to \R$ by setting
  \begin{equation*}
    g_n(\iii) = \sum_{\jjj \in \Sigma_n} \min\{ g(\hhh) : \hhh \in [\jjj] \} \cdot \chi_{[\jjj]}
  \end{equation*}
  for all $\iii \in \Sigma$. Observe that $(g_n)_{n \in \N}$ is an increasing sequence of continuous functions such that $g_n \to g$ as $n \to \infty$. Thus, if we are able to show that there exists a distribution $\hat Q$ on $\PP(\Sigma)$ such that
  \begin{equation*}
    \int_{\PP_\Sigma} f(\mu) g_n(\iii) \dd MQ(\mu,\iii) = \int_{\PP(\Sigma)} f(\mu) \int_\Sigma g_n(\iii) \dd\mu(\iii) \dd\hat Q(\mu)
  \end{equation*}
  for all $n \in \N$, then the claim follows from the monotone convergence theorem and Lemma \ref{thm:adapted2}.
  
  Fix $n \in \N$ and define a linear operator $T \colon C(\PP(\Sigma)) \to C(\PP(\Sigma))$ by setting
  \begin{equation*}
    Th(\mu) = \sum_{i \in \{1,\ldots,N\}} \mu([i]) h(\mu_i)
  \end{equation*}
  for all $h \in C(\PP(\Sigma))$. By the Riesz representation theorem, for every positive continuous linear function $\Lambda \colon C(\PP(\Sigma)) \to \R$ with dual norm one, there exists a unique $\overline{Q} \in \PP(\PP(\Sigma))$ such that $\Lambda(h) = \int_{\PP(\Sigma)} h(\mu) \dd\overline{Q}(\mu)$ for all $h \in C(\PP(\Sigma))$. Thus there exists an adjoint operator $T^* \colon \PP(\PP(\Sigma)) \to \PP(\PP(\Sigma))$ such that
  \begin{equation} \label{eq:adjoint}
    \int_{\PP(\Sigma)} Th(\mu) \dd\overline{Q}(\mu) = \int_{\PP(\Sigma)} h(\mu) \dd T^*\overline{Q}(\mu)
  \end{equation}
  for all $h \in C(\PP(\Sigma))$ and $\overline{Q} \in \PP(\PP(\Sigma))$; see e.g.\ \cite[Theorem 2.14]{Rudin1966} and \cite[Theorem 4.10]{Rudin1973}. Observe that to apply the Riesz representation theorem, we momentarily worked with signed measures. Let us show that $T^*$ acts on $\PP(\PP(\Sigma))$. Fix $\overline{Q} \in \PP(\PP(\Sigma))$. If there are $A \subset \PP(\Sigma)$ and $\delta>0$ so that $T^*\overline{Q}(A)=-\delta$, then, by choosing a positive function $h \in C(\PP(\Sigma))$ such that $\int_{\PP(\Sigma)} |h(\mu)-\chi_A(\mu)| \dd T^*\overline{Q}(\mu) \le \delta/2$, we see that
  \begin{equation*}
    \delta \le \int_{\PP(\Sigma)} Th(\mu) \dd\overline{Q}(\mu) - T^*\overline{Q}(A) 
    \le \int_{\PP(\Sigma)} |h(\mu)-\chi_A(\mu)| \dd T^*\overline{Q}(\mu) \le \delta/2.
  \end{equation*}
  This contradiction shows that we may restrict our analysis to measures.
  
  Let $h \in C(\PP(\Sigma))$ be so that
  \begin{equation*}
    h(\mu) = f(\mu) \int_\Sigma g_n(\iii) \dd\mu(\iii)
  \end{equation*}
  for all $\mu \in \PP(\Sigma)$. Recalling that $Q$ is adapted, we choose $\overline{Q} \in \PP(\PP(\Sigma))$ so that it satisfies \eqref{eq:adapted}. Since each $\mu \in \PP(\Sigma)$ is a measure on $\Sigma$ with $\spt(\mu) \subset \Sigma$ we get
  \begin{equation} \label{eq:lasku}
  \begin{split}
    Th(\mu) &= \sum_{i \in \{1,\ldots,N\}} \mu([i]) h(\mu_i) \\
    &= \sum_{i \in \{1,\ldots,N\}} \mu([i]) f(\mu_i) \sum_{\jjj \in \Sigma_n} \min\{ g(\hhh) : \hhh \in [\jjj] \} \cdot \mu_i([\jjj]) \\
    &= \sum_{i \in \{1,\ldots,N\}} f(\mu_i) \sum_{\jjj \in \Sigma_n} \min\{ g(\hhh) : \hhh \in [\jjj] \} \cdot \mu([i\jjj]) \\
    &= \sum_{i \in \{1,\ldots,N\}} f(\mu_i) \int_{[i]} g_n(\sigma(\iii)) \dd\mu(\iii) \\
    &= \int_\Sigma f(\mu_{\iii|_1}) g(\sigma(\iii)) \dd\mu(\iii)
  \end{split}
  \end{equation}
  for all $\mu \in \PP(\Sigma)$.
  Observe that since $\mu_n \to \mu$ weakly if and only if $\mu_n([\iii]) \to \mu([\iii])$ for all $\iii \in \Sigma_*$ the function $\mu \mapsto f(\mu_i)$ is continuous whenever $\mu([i])>0$. Thus the bounded function $(\mu,\iii) \mapsto f(\mu_{\iii|_1}) g(\sigma(\iii))$ defined on $\PP_\Sigma$ is measurable since its restriction to $\Omega$ is continuous and $\Omega$ has full measure by Lemma \ref{thm:adapted_omega}. Now, by Lemma \ref{thm:adapted_def}, \eqref{eq:lasku}, and \eqref{eq:adjoint}, we get
  \begin{align*}
    \int_{\PP_\Sigma} f(\mu) g_n(\iii) \dd MQ(\mu,\iii) &= \int_{\PP_\Sigma} f(\mu_{\iii|_1}) g(\sigma(\iii)) \dd Q(\mu,\iii) \\
    &= \int_{\PP(\Sigma)} \int_\Sigma f(\mu_{\iii|_1}) g(\sigma(\iii)) \dd\mu(\iii) \dd\overline{Q}(\mu) \\
    &= \int_{\PP(\Sigma)} Th(\mu) \dd\overline{Q}(\mu)
    = \int_{\PP(\Sigma)} h(\mu) \dd T^*\overline{Q}(\mu) \\
    &= \int_{\PP(\Sigma)} f(\mu) \int_\Sigma g_n(\iii) \dd\mu(\iii) \dd T^*\overline{Q}(\mu).
  \end{align*}
  This is what we wanted to show.
\end{proof}

%
%
%

If $\mu \in \PP(\Sigma)$ and $\iii \in \Sigma$, then we define
\begin{equation*}
  \la \mu,\iii \ra_n = \tfrac{1}{n} \sum_{k=0}^{n-1} \delta_{M^k(\mu,\iii)}
\end{equation*}
for all $n \in \N$. Any accumulation point of $\{ \la \mu,\iii \ra_n \}_{n \in \N}$ is called a \emph{micromeasure distribution} of $\mu$ at $\iii$. The family of micromeasure distributions of $\mu$ at $\iii$ is denoted by $\MD(\mu,\iii)$.
We say that $\mu \in \PP(\Sigma)$ \emph{generates} a distribution $Q$ on $\PP_\Sigma$ at $\iii$ if $\MD(\mu,\iii) = \{ Q \}$. If $\mu$ generates $Q$ at $\mu$-almost all $\iii$, then we say that $\mu$ \emph{generates} $Q$.

\begin{lemma} \label{thm:generated-adapted}
  If $Q \in \MD(\mu,\iii)$ is a micromeasure distribution on $\PP_\Sigma$, then $Q$ is adapted for $\mu$-almost all $\iii \in \Sigma$.
\end{lemma}

\begin{proof}
  Let $\FF_m$ be the $\sigma$-algebra generated by the $m$-th level cylinders. Fix $f \in C(\PP(\Sigma))$, $g \in \{ h \in C(\Sigma) : h \text{ is $\FF_m$-measurable for some } m \}$, and write
  \begin{equation*}
    X_k(\iii) = f(\mu_{\iii|_k})\biggl( \int_\Sigma g(\hhh) \dd \mu_{\iii|_k}(\hhh) - g(\sigma^k(\iii)) \biggr)
  \end{equation*}
  for all $\iii \in \Sigma$ and $k \in \N$. Since
  \begin{align*}
    \mathbb{E}\bigl( f(\mu_{\iii|_k})g(\sigma^k(\iii)) \,\big|\, \FF_{k-1} \bigr) &= \frac{1}{\mu([\iii|_{k-1}])} \sum_{h \in \{ 1,\ldots,N \}} \sum_{\jjj \in \Sigma_m} f(\mu_{\iii|_{k-1}h}) g(\sigma^k(\iii)) \mu([\iii|_{k-1}h\jjj]) \\
    &= \frac{1}{\mu([\iii|_{k-1}])} \sum_{h \in \{ 1,\ldots,N \}} f(\mu_{\iii|_{k-1}h}) \int_\Sigma g(\jjj) \dd \mu_{\iii|_{k-1}h}(\jjj) \mu([\iii|_{k-1}h]) \\
    &= \mathbb{E}\biggl( f(\mu_{\iii|_k}) \int_\Sigma g(\iii) \dd \mu_{\iii|_k}(\jjj) \,\bigg|\, \FF_{k-1} \biggr)
  \end{align*}
  we notice that $\mathbb{E}(X_k \,|\, \FF_{k-1}) = 0$ for all $k \in \N$. It is easy to see that $\sup_{k \in \N} \|X_k\|_{L^2(\mu)} < \infty$. Therefore, by \cite[Chapter VII.9, Theorem 3]{Feller1971}, we have $\tfrac{1}{n}\sum_{k=0}^{n-1} X_k(\iii) = 0$ for $\mu$-almost all $\iii \in \Sigma$.

  Since $Q = \lim_{n \to \infty} \tfrac{1}{n} \sum_{k=0}^{n-1} \delta_{M^k(\mu,\iii)}$, where the convergence is along a subsequence, and $\iii$ is chosen according to $\mu$, we thus have
  \begin{equation} \label{eq:integraali-nolla}
    0 = \lim_{n \to \infty} \tfrac{1}{n}\sum_{k=0}^{n-1} X_k(\iii) = \int_{\PP_\Sigma} f(\nu) \biggl( \int_\Sigma g(\hhh) \dd \nu(\hhh) - g(\jjj) \biggr) \dd Q(\nu,\jjj).
  \end{equation}
  The function space $C(\Sigma)$ is separable in the uniform norm and, furthermore, each element in the dense set separating $C(\Sigma)$ can be chosen to be $\FF_m$-measurable for some $m$. Thus the equation \eqref{eq:integraali-nolla} holds for all $f \in C(\PP(\Sigma))$ and $g \in C(\Sigma)$. Let $\overline{Q} \in \PP(\Sigma)$ be the projection of $Q$ onto the measure component. Then
  \begin{align*}
    \int_{\PP(\Sigma)} f(\nu) \int_\Sigma g(\hhh) \dd \nu(\hhh) \dd \overline{Q}(\nu) &= \int_{\PP_\Sigma} f(\nu) \int_\Sigma g(\hhh) \dd \nu(\hhh) \dd Q(\nu,\jjj) \\
    &= \int_{\PP_\Sigma} f(\nu) g(\jjj) \dd Q(\nu,\jjj)
  \end{align*}
  and Lemma \ref{thm:adapted2} finishes the proof.
\end{proof}

Let $Q$ be a distribution on $\PP_\Sigma$. Recall that $A \subset \PP_\Sigma$ is \emph{invariant} (with respect to $Q$) if $Q(M^{-1}(A) \triangle A) = 0$. It follows from Lemmas \ref{thm:invariant_adapted} and \ref{thm:adapted_omega} that $\Omega$ is invariant with respect to any adapted distribution. Furthermore, recall that $Q$ is \emph{invariant} if $MQ=Q$ and \emph{ergodic} if $Q(A) \in \{ 0,1 \}$ for all invariant $A \subset \PP_\Sigma$. An invariant and adapted distribution $Q$ on $\PP_\Sigma$ is called a \emph{symbolic CP-distribution}.

\begin{lemma} \label{thm:generated-CP}
  If $Q \in \MD(\mu,\iii)$ is a micromeasure distribution on $\PP_\Sigma$, then $Q$ is a symbolic CP-distribution for $\mu$-almost all $\iii \in \Sigma$.
\end{lemma}

\begin{proof}
  The adaptedness follows from Lemma \ref{thm:generated-adapted}. Therefore, it suffices to show that an adapted micromeasure distribution $Q$ is invariant. By Lemma \ref{thm:adapted_omega}, $Q(\Omega)=1$ where $\Omega$ is defined in \eqref{eq:omegagamma}. Since $Q = \lim_{n \to \infty}\la \mu,\iii \ra_n$, where the convergence is along a subsequence, and, by the remark made before Lemma \ref{thm:invariant_adapted}, $M$ is continuous on $\Omega$, we have
  \begin{align*}
    MQ-Q &= \lim_{n \to \infty} (M\la \mu,\iii \ra_n - \la \mu,\iii \ra_n) \\
    &= \lim_{n \to \infty} \tfrac{1}{n} \biggl( \sum_{k=0}^{n-1} M\delta_{M^k(\mu,\iii)} - \sum_{k=0}^{n-1} \delta_{M^k(\mu,\iii)} \biggr) \\
    &= \lim_{n \to \infty} \tfrac{1}{n} (\delta_{M^n(\mu,\iii)} - \delta_{(\mu,\iii)}) = 0,
  \end{align*}
  which is what we wanted to show.
\end{proof}

\begin{lemma} \label{thm:lemma18}
  If $Q$ is a symbolic CP-distribution on $\PP_\Sigma$, then its ergodic components are symbolic $CP$-distributions.
\end{lemma}

\begin{proof}
  By the ergodic decomposition, each ergodic component $Q_\omega$ is invariant. Thus it suffices to show that an ergodic and invariant distribution is adapted. If $Q$ is such a distribution on $\PP_\Sigma$ and $f$ is a continuous function defined on $\PP_\Sigma$, then, by the Birkhoff ergodic theorem, we have
  \begin{equation*}
    \lim_{n \to \infty} \int_{\PP_\Sigma} f(\nu,\jjj) \dd \la \mu,\iii \ra_n(\nu,\jjj) = \lim_{n \to \infty} \tfrac{1}{n} \sum_{k=0}^{n-1} f(M^k(\mu,\iii)) = \int_{\PP_\Sigma} f(\nu,\jjj) \dd Q(\nu,\jjj)
  \end{equation*}
  for $Q$-almost all $(\mu,\iii)$. Since this holds simultaneously for any countable collection of continuous functions and $C(\PP_\Sigma)$ is separable in the uniform norm, it holds simultaneously for all continuous $f$. Therefore $\overline{Q}$-almost every $\mu$ generates $Q$ where $\overline{Q} \in \PP(\Sigma)$ is the projection of $Q$ onto the measure component. By Lemma \ref{thm:generated-adapted}, $Q$ is adapted.
\end{proof}

\subsection{CP-distributions}
Let
\begin{align*}
  \fii_1 \colon \R^2 \to \R^2, \quad &\fii_1(x,y) = \tfrac12(x,y) + (\tfrac12,\tfrac12), \\
  \fii_2 \colon \R^2 \to \R^2, \quad &\fii_2(x,y) = \tfrac12(x,y) + (\tfrac12,-\tfrac12), \\
  \fii_3 \colon \R^2 \to \R^2, \quad &\fii_3(x,y) = \tfrac12(x,y) + (-\tfrac12,-\tfrac12), \\
  \fii_4 \colon \R^2 \to \R^2, \quad &\fii_4(x,y) = \tfrac12(x,y) + (-\tfrac12,\tfrac12).
\end{align*}
Observe that $[-1,1]^2 = \bigcup_{i=1}^4 \fii_i([-1,1]^2)$ and the interiors of the cubes $\fii_i([-1,1]^2)$ and $\fii_j([-1,1]^2)$ are disjoint whenever $i \ne j$. Write $\fii_\iii = \fii_{i_1} \circ \cdots \circ \fii_{i_n}$ for all $\iii = i_1 \cdots i_n \in \Sigma_n$ and $n \in \N$. Then, similarly, $[-1,1]^2 = \bigcup_{\iii \in \Sigma_n} \fii_\iii([-1,1]^2)$ and the interiors of the dyadic cubes $\fii_\iii([-1,1]^2)$ and $\fii_\jjj([-1,1]^2)$ of side length $2^{1-n}$ are disjoint for all $\iii,\jjj \in \Sigma_n$ with $\iii \ne \jjj$.

Each element in $\Sigma$ can now be made to represent a point in $[-1,1]^2$. The \emph{projection} $\pi \colon \Sigma \to [-1,1]^2$ is defined by setting $\pi(\iii)$ to be the single point in the set $\bigcap_{n \in \N} \fii_{\iii|_n}([-1,1]^2)$ for all $\iii \in \Sigma$. The projection is continuous and surjective. Furthermore, each point in $[-1,1]^2 \setminus \Delta$ is represented by a unique element in $\Sigma$, where $\Delta = \bigcup_{\iii \in \Sigma_*} \partial \fii_\iii([-1,1]^2)$ and $\partial A$ is the boundary of the set $A$.

If $Q$ is a distribution on $\PP_\Sigma$ and $\overline{Q} \in \PP(\Sigma)$ its projection onto the measure component, then the \emph{intensity measure} $[Q]$ of $Q$ is defined by
\begin{equation*}
  [Q](A) = \int_{\PP(\Sigma)} \pi\mu(A) \dd \overline{Q}(\mu)
\end{equation*}
for all $A \subset [-1,1]^2$. Write $\LL$ to denote the normalized Lebesgue measure on $[-1,1]^2$. We say that a distribution $Q$ has \emph{Lebesgue intensity} if $[Q]=\LL$. If a distribution $Q$ on $\PP_\Sigma$ has Lebesgue intensity, then $[Q](\Delta) = \LL(\Delta) = 0$ and we have $\pi\nu(\Delta) = 0$ for $\overline{Q}$-almost all $\nu$. 

\begin{lemma} \label{thm:trans-leb-intensity}
  Let $\mu' \in \PP([-\tfrac12,\tfrac12]^2)$, write $\mu_z' = T_z\mu'$ for all $z \in [-\tfrac12,\tfrac12]^2$, and let $\mu_z$ be a measure on $\Sigma$ such that $\pi\mu_z = \mu_z'$. Then, for $\mu_z$-almost every $\iii_z \in \Sigma$, a micromeasure distribution $Q_z \in \MD(\mu_z,\iii_z)$ has Lebesgue intensity for $\LL$-almost all $z \in [-\tfrac12,\tfrac12]^2$.
\end{lemma}

\begin{proof}
  Let $g \in C([-1,1]^2)$ and pick $\mu_0$-generic $\iii \in \Sigma$. This means that the point $\iii_z$ for which $\pi(\iii_z) = \pi(\iii) - z$ is $\mu_z$-generic for all $z \in [-\tfrac12,\tfrac12]$. Therefore, by Lemma \ref{thm:generated-adapted}, $Q_z = \lim_{n \to \infty} \la \mu_z,\iii_z \ra_n$ is adapted. Note that here the convergence is along a subsequence. Furthermore, if $z$ is chosen according to the Lebesgue measure, then $\iii_z$ is generic for the uniform distribution $\lambda$ on $\Sigma$. Hence, by Lemma \ref{thm:adapted2} and the Birkhoff ergodic theorem, we have
  \begin{align*}
    \int_{\PP(\Sigma)} \int_\Sigma g(x) \dd \pi\nu(x) \dd \overline{Q}_z(\nu)
    &= \int_{\PP_\Sigma} g \circ \pi(\jjj) \dd Q_z(\nu,\jjj) \\
    &= \lim_{n \to \infty} \int_{\PP_\Sigma} g \circ \pi(\jjj) \dd \la \mu_z,\iii_z \ra_n(\nu,\jjj) \\
    &= \lim_{n \to \infty} \tfrac{1}{n} \sum_{k=0}^{n-1} g \circ \pi(\sigma^k(\iii_z)) \\
    &= \int_\Sigma g \circ \pi(\jjj) \dd \lambda(\jjj)
    = \int_{[-1,1]^2} g(x) \dd \LL(x).
  \end{align*}
  It follows that $[Q_z] = \LL$ for $\LL$-almost all $z \in [-\tfrac12,\tfrac12]^2$.
\end{proof}

Observe that in Lemma \ref{thm:trans-leb-intensity}, instead of translating the measure $\mu$, one can slightly adjust the dyadic grid to see that $\pi\nu(\Delta)=0$ for $\overline{Q}$-almost all $\nu$, where $Q$ is a micromeasure distribution of $\mu$. Therefore, up to choosing a suitable dyadic grid, the previous lemmas guarantee that all the symbolic results hold analogously in $[-1,1]^2$. More precisely, let $\II_n = \{ [\tfrac{i}{2^n},\tfrac{i+2}{2^n}) : i \in \N \text{ is even} \}$ and define $\DD_n = \{ I_1 \times I_2 : I_1,I_2 \in \II_n \}$ for all $n \in \N$. Note that for each $\iii \in \Sigma_n$ there exists $D \in \DD_n$ such that $\fii_\iii([-1,1]^2) = \overline{D}$. Let $D_n(x) \in \DD_n$ be the only element containing $x \in \R^2$. If $x \in [-1,1]^2 \setminus \Delta$, then there exists unique $\iii \in \Sigma$ such that $\pi(\iii) = x$ and $\overline{D_n(x)} = \fii_{\iii|_n}([-1,1]^2)$ for all $n \in \N$. Let $B$ be a closed ball and write $T_B \colon \R^2 \to \R^2$ for the orientation preserving homothety map taking $B$ onto $[-1,1]^2$. Define $M \colon \MM \times [-1,1]^2 \to \MM \times [-1,1]^2$ by setting
\begin{equation*}
  M(\mu,x) = (T_{D_1(x)}\mu, T_{D_1(x)}x).
\end{equation*}
Now $M^\square$ is a mapping from $\MM^\square \times [-1,1]^2$ to $\MM^\square \times [-1,1]^2$ such that
\begin{equation*}
  M^\square(\mu,x) = (T^\square_{D_1(x)}\mu, T_{D_1(x)}x) = \biggl( \frac{T_{D_1(x)}\mu}{\mu(D_1(x))}, T_{D_1(x)}(x) \biggr).
\end{equation*}
This is well-defined whenever $\mu(D_1(x))>0$. Observe that the map $T \colon [-1,1]^2 \to [-1,1]^2$ defined by $T(x) = T_{D_1(x)}(x)$ for all $x \in [-1,1]^2$ satisfies $T \circ \pi(\iii) = \pi \circ \sigma(\iii)$ for all $\iii \in \Sigma \setminus \pi^{-1}(\Delta)$. We adopt all the definitions from the symbolic case to be used analogously with $M^\square$ in $[-1,1]^2$. In particular, we say that a distribution $Q$ on $\PP$ is a \emph{CP-distribution} if it is invariant and adapted. Via this correspondence, Lemma \ref{thm:generated-CP} thus is Theorem \ref{thm:theorem11} in the CP-distribution setting.

The final lemma in this section guarantees that also the Lebesgue intensity is preserved when going to ergodic components.

\begin{lemma}
  If a distribution $Q$ on $\PP_\Sigma$ has Lebesgue intensity, then $Q$-almost every ergodic component of $Q$ has Lebesgue intensity.
\end{lemma}

\begin{proof}
  Observe first that, by the ergodic decomposition, $[Q](\Delta)=0$ implies $[Q_\omega](\Delta) = 0$ for $Q$-almost all $\omega$. Therefore, for such $\omega$, the measure-theoretic entropy $h_{[Q_\omega]}(T)$ satisfies $h_{[Q_\omega]}(T) \le \log 4$, with equality if and only if $[Q_\omega] = \LL$. By the affinity of the entropy,
  \begin{equation*}
    h_{[Q]}(T) = \int_{\PP_\Sigma} h_{[Q_\omega]}(T) \dd Q(\omega)
  \end{equation*}
  which finishes the proof.
\end{proof}

\subsection{Factors, extensions, and suspensions}
Given two measure preserving systems $(X,Q,M)$ and $(X',Q',M')$, a map $\pi \colon X \to X'$ is called a \emph{factor map} if the underlying semigroups coincide, $\pi Q = Q'$, and $\pi$ intertwines the actions of the semigroups, i.e.\ $\pi \circ M_s = M_s' \circ \pi$ for all $s$. In this case, we also say that $(X',Q',M')$ is a factor of $(X,Q,M)$. The factor of an ergodic system is ergodic.

Let $(X,P)$ be a metric probability space, and consider the product spaces $(X^\N,P^\N)$ and $(X^\Z,P^\Z)$. The shift map $\sigma$ defined by $\sigma((x_i)_i) = (x_{i+1})_i$ acts on both spaces. If $(X^\N,Q,M)$ is a measure preserving system, then there exists a measure preserving system $(X^\Z,\hat Q,\hat M)$ such that $(X^\N,Q,M)$ is a factor of $(X^\Z,\hat Q,\hat M)$ under the natural projection map. This is called the \emph{two-sided extension} of the one-sided system $(X^\N,Q,M)$. The two-sided extension is ergodic if and only if the one-sided system is ergodic. Furthermore, any discrete measure preserving system $(X,Q,M)$ can always be represented as a shift space via the identification $x \mapsto (M^kx)_{i=0}^\infty$. The measure on $X^\N$ is the push-forward of $\mu$ under this map.

A standard way to build measure preserving flows from discrete systems is via \emph{suspensions}. Let $(X,Q,M)$ be a discrete measure preserving system and let $r>0$. Write $X' = X \times [0,r)$, $P = Q \times \lambda$, where $\lambda$ is the normalized Lebesgue measure on $[0,r)$, and set $T_t(x,s) = (x,s+t)$ if $s+t<r$ and $T_t(x,r-t) = (M(x),0)$. By iterating, this defines a flow, a so called \emph{suspension flow}, which indeed preserves the measure $P$. Moreover, the measure preserving system $(X,Q,M)$ is ergodic if and only if its suspension flow $(X',P,T)$ is ergodic.

\subsection{Proof of Theorem \ref{thm:theorem11}} \label{sec:MH-proof}
Let us next turn Lemma \ref{thm:generated-CP} into the language of fractal distributions and prove Theorem \ref{thm:theorem11}. We shall start by constructing an extended version of a CP-distribution. Let $Q$ is a CP-distribution. Its extended version $\hat Q$ is an $M^*$-invariant distribution on $\MM^* \times [-1,1]^2$ whose push-forward under $(\mu,x) \mapsto (\mu^\square,x)$ is $Q$. Let $(\mu_n,x_n)_{n \in \Z}$ be a stationary process via marginal $Q$ such that
\begin{equation*}
  (\mu_{n+1},x_{n+1}) = M^\square(\mu_n,x_n).
\end{equation*}
Let
\begin{equation*}
  \nu_n = T^*_{D_n(x_{-n})}\mu_{-n} \quad \text{and} \quad
  E_n = T_{D_n(x_{-n})}([-1,1]^2)
\end{equation*}
for all $n \in \N \cup \{ 0 \}$. Now $\nu_n$ is supported on $E_n$ and $[-1,1]^2 = E_0 \subset E_1 \subset E_2 \subset \cdots$. For $n \ge m \ge 0$ we have
\begin{equation*}
  (M^\square)^m(\mu_{-n},x_{-n}) = (\mu_{-n+m},x_{-n+m}).
\end{equation*}
This means that $\nu_n|_{E_m} = \nu_m$ and we may define a Radon measure $\nu$ by setting
\begin{equation*}
  \nu(A) = \lim_{n \to \infty} \nu_n(A)
\end{equation*}
for all Borel sets $A \subset \R^2$. Let the distribution of $\nu$ be $\overline{\hat Q}$, and let $\hat Q$ be such that $\dd\hat Q(\nu,x) = \dd\nu(x)\dd\overline{\hat Q}(\nu)$. The distribution $\hat Q$ on $\MM^* \times [-1,1]^2$ is $M^*$-invariant. It is easy to see that $(\mu,x) \mapsto (\mu^\square,x)$ is a factor map between the measure preserving systems $(\MM^* \times [-1,1]^2,\hat Q,M^*)$ and $(\MM^\square \times [-1,1]^2,Q,M^\square)$. The distribution $\hat Q$ is called the \emph{extended version} of $Q$.

If $\mu \in \MM$ and $x \in \spt(\mu)$, then we write
\begin{equation*}
  \cent_0(\mu,x) = T^*_x\mu.
\end{equation*}
The \emph{centering} of an extended CP-distribution $\hat Q$ is the distribution
\begin{equation*}
  \cent\hat Q = \frac{1}{\log 2} \int_0^{\log 2} S^*_t \cent_0\hat Q \dd t,
\end{equation*}
where
\begin{equation*}
  \cent_0 \hat Q = \int \la \mu \ra^*_{[-1,1]^2} \dd\overline{\hat Q}(\mu).
\end{equation*}
In other words, the centering of $\hat Q$ is a push-forward of the suspension flow of $\hat Q$ under the magnification of $\mu$ at $x$. Indeed, it is immediate that $(\mu',x') = M^*(\mu,x)$ satisfies
\begin{equation*}
  S^*_{\log 2}(T_x\mu) = T^*_{x'}\mu'.
\end{equation*}
Therefore, the map $\cent_0(\mu,x) = T^*_x\mu$ is a factor map from the discrete measure preserving system $(\MM^* \times [-1,1]^2,\hat Q,M^*)$ to the discrete system $(\MM^*,\cent_0\hat Q,S^*_{\log 2})$. The flow $(\MM^*,\cent\hat Q,S^*)$ is thus a factor of the suspension of $(\MM^* \times [-1,1]^2,\hat Q,M^*)$.

\begin{lemma} \label{thm:cent-cont}
  The mapping $Q \mapsto (\cent\hat Q)^\square$ defined on the set of all CP-distributions is continuous.
\end{lemma}

\begin{proof}
  Since $\cent_0(\mu,x)$ depends only on the restriction of $\mu$ to $[-2,2]^2$, we can work with push-forward of $Q$ under the normalization to $[-2,2]^2$ instead of full extended version of $Q$. We shall, however, omit the details concerning this and naively think that, at every step, we are working with CP-distributions.
  
  Let $f \in C(\MM^\square)$. By the definition of the centering operation and Fubini's theorem, we have
  \begin{equation*}
    \int_{\MM^\square} f(\mu) \dd (\cent\hat Q)^\square(\mu) = \frac{1}{\log 2} \int_{\MM^\square \times [-1,1]^2} \int_0^{\log 2} f(S_t^\square T^*_x\mu) \dd t \dd Q(\mu,x).
  \end{equation*}
  The maps $S_t^\square$ are discontinuous at measures $\nu$ with $\nu(\partial [-e^{-t},e^{-t}]^2) > 0$. Note that a given measure $\nu$ can charge the boundary of a ball centered at the origin only at most countably many radii. It follows from the dominated convergence theorem that the inner integral is a continuous function of $T^*_x\mu$, which in turn is a continuous function of $(\mu,x)$. The lemma follows.
\end{proof}

\begin{theorem} \label{thm:cent-FD}
  If $Q \in \MD(\mu,x)$ is a CP-distribution on $\MM^\square$, then $\cent\hat Q$ is a fractal distribution on $\MM^*$.
\end{theorem}

\begin{proof}
  Let us first show that $R = \cent_0\hat Q$ is quasi-Palm. Let $(\mu_n,x_n)_{n \in \Z}$ be the two-sided stationary process with marginal $Q$ and $(\mu_{n+1},x_{n+1}) = M^\square(\mu_n,x_n)$. If
  \begin{equation} \label{eq:nu-map}
    \nu = \lim_{n \to \infty} T^*_{D_n(x_{-n})}\mu_{-n},
  \end{equation}
  then we know that $\nu \sim \hat Q$. Let $B \subset \R^2$ be a bounded neighborhood of $0$ and define
  \begin{equation*}
    R' = \int \la \nu \ra^*_B \dd R(\nu) = \int \biggl( \int_B \delta_{T^*_x\nu} \biggr) \dd \nu(x) \dd R(\nu).
  \end{equation*}
  We must show that $R \sim R'$.
  
  Fix a measurable $\UU \subset \MM^*$ and assume that $R(\UU)=0$. We want to show that $R'(\UU)=0$. Observe that it suffices to prove
  \begin{equation*}
    \iint \chi_{\UU}(T^*_x\nu) \dd \nu(x) \dd R(\nu) = 0.
  \end{equation*}
  For this it is enough to show that for almost every realization $(\mu_n,x_n)_{n \in \Z}$ of the process and $\nu$ satisfying \eqref{eq:nu-map}, we have
  \begin{equation*}
    \int \chi_{\UU}(T^*_x\nu) \dd \nu(x) = 0.
  \end{equation*}
  By \eqref{eq:nu-map}, this will follow if for almost every realization of the process and every $n \in \N$, we have
  \begin{equation*}
    \int \chi_{\UU}(T^*_x\nu) \dd T^*_{D_n(x_{-n})} \mu_{-n}(x) = 0.
  \end{equation*}
  To show this, fix $n \in \N$. By stationarity of the process $(\mu_n,x_n)_{n \in \Z}$ and the fact that the map $(\mu_n,x_n)_{n \in \Z} \mapsto \nu$ intertwines the shift operation and $S^*_{\log 2}$, we have
  \begin{equation*}
    \int \chi_{\UU}(T^*_x\nu) \dd T^*_{D_n(x_{-n})}\mu_{-n}(x) = \int \chi_{\UU}(S^*_{n\log 2}T_x\nu) \dd\mu_0(x) = \int \chi_{S^*_{-n\log 2}\UU}(T_x\nu) \dd\mu_0(x).
  \end{equation*}
  Since $R = \int \la \mu_0 \ra^*_{[-1,1]^2} \dd \overline{\hat Q}(\mu_0)$ we have
  \begin{equation*}
    \iint \chi_{S^*_{-n\log 2}\UU}(T_x\nu) \dd\mu_0(x) \dd\overline{\hat Q}(\mu_0) = R(S^*_{-n\log 2}\UU) = R(\UU) = 0
  \end{equation*}
  and hence,
  \begin{equation*}
    \int \chi_{S^*_{-n\log 2}\UU}(T_x\nu) \dd\mu_0(x) = 0
  \end{equation*}
  for $\overline{\hat Q}$-almost all $\mu_0$ which is what we wanted.

  Conversely, for any measure $\nu$, since $0 \in B$ we have
  \begin{equation*}
    \la \nu \ra^*_{[-1,1]^2} \ll \int \la \theta \ra^*_B \dd\la \nu \ra^*_{[-1,1]^2}(\theta).
  \end{equation*}
  Since $R = \int \la \nu \ra^*_{[-1,1]^2} \dd \overline{\hat Q}(\nu)$ we get
  \begin{equation*}
    R \ll \iint \la \theta \ra^*_B \dd\la \nu \ra^*_{[-1,1]^2}(\theta) \dd \overline{\hat Q}(\nu) = \int \la \theta \ra^*_B \dd R(\theta) = R'.
  \end{equation*}
  Finally, $\cent Q = \tfrac{1}{\log 2} \int_0^{\log 2} S^*_t \cent_0 Q \dd t$ is quasi-Palm since $R = \cent_0Q$ has the same property.
\end{proof}

Theorem \ref{thm:theorem11} follows immediately from Theorem \ref{thm:cent-FD} and the following result.

\begin{theorem} \label{thm:theorem19}
  If $\mu \in \PP([-\tfrac12,\tfrac12]^2)$, then for $\mu$-almost every $x \in \R^2$ and for every $P \in \TD(\mu,x)$ there exists a CP-distribution $Q_z \in \MD(T_z\mu,x-z)$ such that $P = (\cent\hat Q_z)^\square$ for $\LL$-almost all $z \in [-\tfrac12,\tfrac12]^2$.
\end{theorem}

\begin{proof}
  Choose $x \sim \mu$. Since $P$ is an accumulation point of $\{ \la \mu \ra^\square_{x,T} \}_{T \in \R_+}$, there exists a subsequence so that Lemmas \ref{thm:trans-leb-intensity} and \ref{thm:generated-CP} imply that the micromeasure distribution $Q = \lim_{i \to \infty} \la T_z\mu,x-z \ra_{n_i}$ is a CP-distribution for $\LL$-almost all $z \in [-\tfrac12,\tfrac12]^2$. Observe also that $\{ \la \mu \ra^\square_{x,T} \}_{T \in \R_+} = \{ \la T_z\mu \ra^\square_{x-z,T} \}_{T \in \R_+}$ for all $z \in \R^2$. We shall omit the translation by $z$ for notational convenience. We will prove that the extension of $Q$ has the claimed properties.
  
  Let $P_m = S_{m\log 2}^\square \cent Q$ for all $m \in \N$. This means that if
  \begin{equation*}
    g_m(\nu,y) = \int_0^1 \delta_{S^\square_{(m+t)\log 2}T_y\nu} \dd t,
  \end{equation*}
  then $P_m = g_mQ$. Recall that the maps $S^\square_{t\log 2}$ are discontinuous at measures $\nu$ which satisfy $\nu(\partial [-e^{-t\log 2},e^{-t\log 2}]^2)>0$. A given measure $\nu$ can charge the boundary of a ball centered at the origin only at most countably many radii. Therefore, if $(\nu_n,y_n) \to (\nu,y)$, then $S^\square_{(m+t)\log 2}T_{y_n}\nu_n \to S^\square_{(m+t)\log 2}T_y\nu$ for all except at most countably many $t$ and hence $g_m(\nu_n,y_n) \to g_m(\nu,y)$.
  
  From the construction of the extended version $\hat Q$ of $Q$ and the definition of the centering operation we see that $P_m \to (\cent\hat Q)^\square$. It therefore suffices to show that $P_m \to P$. To that end, write $P_{m,i} = g_m\la \mu,x \ra_{n_i}$ for all $i \in \N$. The continuity of $g_m$ implies that $P_m = g_mQ = \lim_{i \to \infty}g_m\la \mu,x \ra_{n_i} = \lim_{i \to \infty} P_{m,i}$.

  Fix $m$ and write $(\mu_k,x_k) = M^k(\mu,x)$ for all $k \in \N$. If $x$ is such that $\dist(x,\Delta_0) \ge 2^{-m}$, where $\Delta_0 = \bigcup_{i=1}^4 \partial\fii_i([-1,1]^2)$, then we have
  \begin{equation*}
    S^\square_{(k+m)\log 2}T_x\mu = S^\square_{m\log 2}T_{x_k}\mu_k.
  \end{equation*}
  Thus,
  \begin{equation*}
    \int_0^1 \delta_{S^\square_{(k+m+t)\log 2}T_x\mu} \dd t = S^\square_{m\log 2}\cent\delta_{(\mu_k,x_k)}.
  \end{equation*}
  Now
  \begin{align*}
    P_{m,i} - \la \mu \ra_{x,n_i} &= \tfrac{1}{n_i} \sum_{k=1}^{n_i} S^\square_{m\log 2}\cent\delta_{(\mu_k,x_k)} - \tfrac{1}{n_i} \sum_{k=0}^{n_i-1} \int_0^1 \delta_{S^\square_{t \log 2}T_x\mu} \dd t \\
    &= \tfrac{1}{n_i} \sum_{k=1}^{n_i-m} \biggl( S^\square_{m\log 2}\cent\delta_{(\mu_k,y_k)} - \int_0^1 \delta_{S^\square_{(k+m+t)\log 2}T_x\mu} \dd t \biggr) + \tfrac{m}{n_i} \theta_{m,i},
  \end{align*}
  where $\theta_{m,i}$ is a probability measure. Observe that $\tfrac{m}{n_i} \theta_{m,i} \to 0$ as $i \to \infty$ and in the average, the summands vanish provided that $\dist(x,\Delta_0) \ge 2^{-m}$. Thus, by letting $i \to \infty$, the right-hand side of the above equation is a probability measure whose total mass is asymptotic to
  \begin{equation} \label{eq:TD-FD_asymptotic}
    \tfrac{1}{n_i} \#\{ n \in \{ 1,\ldots,n_i\} : \dist(x_n,\Delta_0) \ge 2^{-m} \}.
  \end{equation}
  Applying the Birkhoff ergodic theorem as in the proof of Lemma \ref{thm:trans-leb-intensity}, we see that
  \begin{equation*}
    \tfrac{1}{n_i} \sum_{k=1}^{n_i} \delta_{x_k} \to \LL
  \end{equation*}
  as $i \to \infty$. Therefore, as $i \to \infty$, \eqref{eq:TD-FD_asymptotic} converges to $1-c_m$, where $c_m \to 0$ as $m \to \infty$. This completes the proof.
\end{proof}

\subsection{Ergodic fractal distributions}
We say that a distribution $P$ on $\MM^*$ is an \emph{ergodic fractal distribution} if it is a fractal distribution and the flow $(\MM^*,P,S^*)$ is ergodic.

Theorem \ref{thm:theorem11} shows that tangent distributions are fractal distributions. The following result shows that ergodic fractal distributions are tangent distributions. The non-ergodic case will be addressed in Theorem \ref{thm:theorem22}.

\begin{theorem} \label{thm:EFD-generated}
  If $P^\square$ is a restricted ergodic fractal distribution on $\MM^\square$, then
  \begin{equation*}
    \TD(\mu,x) = \{ P^\square \}
  \end{equation*}
  for $P^\square$-almost all $\mu \in \MM^\square$ and $\mu$-almost all $x \in \R^2$.
\end{theorem}

\begin{proof}
  The proof of similar to that of Lemma \ref{thm:lemma18}. Let $P$ be the extended version of $P^\square$. For each $f \in C(\PP(\MM^\square))$ let $\hat f \colon \MM^* \to \R$ denote the map $\mu \mapsto f(\mu^\square)$. Observe that this map is no longer continuous, but it is still bounded and measurable. For such a map $f$, the Birkhoff ergodic theorem guarantees that
  \begin{equation} \label{eq:FD-Birkhoff}
    \lim_{T \to \infty} \frac{1}{T} \int_0^T f(S^\square_t\mu) \dd t = \int_{\MM^*} \hat f(\nu) \dd P(\nu)
  \end{equation}
  for $P$-almost all $\mu \in \MM^*$. Since this holds simultaneously for any countable collection of continuous functions and $C(\PP(\MM^\square))$ is separable in the uniform norm, it holds simultaneously for all continuous $f$. By the quasi-Palm property, the limit \eqref{eq:FD-Birkhoff} holds for $\mu = T^*_z\nu$ for $P$-almost all $\nu \in \MM^*$ and for $\nu$-almost all $z \in \R^2$. This is what we wanted to show.
\end{proof}

In Theorem \ref{thm:theorem19}, we showed that for every tangent distribution $P$ there exists a CP-distribution $Q$ such that $P = (\cent\hat Q)^\square$. The following theorem gives the same result for fractal distributions.

\begin{theorem} \label{thm:FDs-are-CPs}
  If $P$ is a fractal distribution on $\MM^*$, then there is a CP-distribution $Q$ such that $\cent\hat Q = P$.
\end{theorem}

\begin{proof}
  If $P$ is an ergodic fractal distribution and $P^\square$ its restricted version, then Theorem \ref{thm:EFD-generated} shows that $\TD(\mu,x) = \{ P^\square \}$ for $P^\square$-almost all $\mu \in \MM^\square$ and for $\mu$-almost all $x \in \R^2$. Therefore, after possibly adjusting the dyadic grid, Theorem \ref{thm:theorem19} implies that there exists an ergodic CP-distribution $Q$ such that $P = \cent\hat Q$.
  
  The centering $\cent$ from the set of extended ergodic CP-distributions to the set of ergodic fractal distributions is thus surjective. Since the associate sets can be verified to be Borel and $\cent$ is continuous by Lemma \ref{thm:cent-cont}, it follows from \cite[Lemma 2.2]{Varadarajan1963} that for each probability measure $\tau$ on the family of all ergodic fractal distributions there is a probability measure $\tau'$ on the family of all ergodic CP-distributions such that $\tau = \cent\tau'$. Since, in particular, this applies to the ergodic decomposition of a given fractal distribution $P$, we have found a CP-distribution $Q$ for which $\cent \hat Q = P$.
\end{proof}

\begin{theorem} \label{thm:theorem10}
  If $P$ is a fractal distribution on $\MM^*$, then $P$-almost every ergodic component of $P$ is an ergodic fractal distribution.
\end{theorem}

\begin{proof}
  Let $P$ be a fractal distribution on $\MM^*$. Recalling Theorem \ref{thm:FDs-are-CPs}, let $Q$ be a CP-distribution whose extension $\hat Q$ satisfies $\cent\hat Q = P$. The ergodic components of $\hat Q$ have discrete centerings that are $S^*_{\log 2}$-invariant and, being factors of ergodic CP-distributions, also ergodic. Since the same holds for the suspensions and the centerings of the ergodic components of $\hat Q$ integrate to $P$, we have, by Theorem \ref{thm:cent-FD}, shown the claim.
%
%
\end{proof}

\begin{theorem} \label{thm:theorem20}
  The set $\{ P^\square : P \text{ is a fractal distribution} \} \subset \PP(\MM^\square)$ is convex and compact.
\end{theorem}

\begin{proof}
  Since the convexity is clear, let $(P^\square_n)_{n = 1}^\infty$ be a sequence of restricted fractal distributions such that $P^\square_n \to P^\square$ in $\PP(\MM^\square)$. It suffices to show that there exists a fractal distribution $P$ such that its restricted version is $P^\square$.
  
  For each $n \in \N$, Theorem \ref{thm:FDs-are-CPs} shows that there exists a CP-distribution $Q_n$ such that $(\cent\hat Q_n)^\square = P^\square_n$. Since, by Lemma \ref{thm:lemma15}, the set of all CP-distributions is compact, the set $\{ Q_n \}_{n \in \N}$ has an accumulation point, say, a CP-distribution $Q$. Since, by Theorem \ref{thm:cent-FD}, $\cent\hat Q$ is a fractal distribution, it suffices to show that $(\cent\hat Q)^\square = P^\square$. But this follows immediately since, by Lemma \ref{thm:cent-cont}, the centering operation is continuous.
\end{proof}

Together with the uniqueness of the ergodic decomposition, Theorem \ref{thm:theorem20} implies that the family of fractal distributions is a Choquet simplex. A \emph{Poulsen simplex} is a Choquet simplex in which extremal points are dense. Recall that by Theorem \ref{thm:theorem10}, extremal points here are the ergodic fractal distributions.

\begin{theorem} \label{thm:theorem21}
  The set $\{ P^\square : P \text{ is a fractal distribution} \} \subset \PP(\MM^\square)$ is a Poulsen simplex.
\end{theorem}

\begin{proof}
  The task is to show that ergodic fractal distributions are weakly dense. By Lemma \ref{thm:cent-cont} and Theorem \ref{thm:FDs-are-CPs}, the centering operation is a continuous surjection from the set of all CP-distributions to the set of all fractal distributions taking ergodic CP-distributions to ergodic fractal distributions. It thus suffices to show that ergodic CP-distributions are weakly dense. We shall do this by first approximating a given CP-distribution by a finite convex combination of ergodic CP-distributions. Then we splice together those finite CP-distributions, constructing a sequence of ergodic CP-distributions converging to the convex combination.
  
  Let us explain what we mean by splicing. Given a sequence $\mathbf{n} = (n_k)_{k = 1}^\infty$ of integers, we define the \emph{splicing map} $\spl_{\mathbf{n}} \colon \Sigma^\N \to \Sigma$ by
  \begin{equation*}
    \spl_{\mathbf{n}}((\iii^k)_{k = 1}^\infty) = \iii^1|_{n_1}\iii^2|_{n_2}\cdots.
  \end{equation*}
  We are interested in splicing product measures $\times_{k=1}^\infty \mu^k$ on $\Sigma^\N$, and also product distributions. If $\nu = \spl_{\textbf{n}} \times_{k=1}^\infty \mu^k$, then, by the definition on splicing, the $\nu$-mass of a cylinder $[\hhh]$ is built from the $\mu^k$-masses of consequtive sub-words of $\hhh$ whose length come from the sequence $\mathbf{n}$. Indeed, if $k \in \N$, $n_1+\cdots+n_k < N \le n_1+\cdots+n_{k+1}$, and $\hhh \in \Sigma_N$, then
  \begin{equation*}
    \nu([\hhh]) = \mu^1([\hhh|_{n_1}]) \mu^2([\sigma^{n_1}(\hhh)|_{n_2}])
    \cdots \mu^k([\sigma^{n_1+\cdots+n_{k-1}}(\hhh)|_{n_k}]) \mu^{k+1}([\sigma^{n_1+\cdots+n_k}(\hhh)]).
  \end{equation*}
  To see this, fix $k \in \N$, $0 \le i \le i+j \le n_{k+1}$, and choose $\kkk \in \Sigma_{n_1+\cdots+n_k+1}$ and $\hhh \in \Sigma_j$. Then we have
  \begin{equation} \label{eq:24}
    \nu_\kkk([\hhh]) = \mu^{k+1}_{\sigma^{n_1+\cdots+n_k}(\kkk)|_i}([\hhh]),
  \end{equation} 
  and, in particular, $\nu_\kkk([\hhh]) = \mu^{k+1}([\hhh])$ when $i=0$. The claim follows by iterating \eqref{eq:24}. Relying on this observation, by choosing a suitable sequence $\mathbf{n}$, we can now control the frequency of occurences of the measures $\mu^k$ in $\la \nu,\iii \ra_N$.

  By the Krein-Milman Theorem, it suffices to show that, given a rational probability vector $(t_1/q,\ldots,t_k/q)$ and ergodic CP-distributions $R_1,\ldots,R_k$, there is a sequence $(Q^N)_{N = 1}^\infty$ of ergodic CP-distributions converging to $\tfrac{1}{q} \sum_{i=1}^k t_iR_i$ as $N \to \infty$. For a fixed $N$, to find the ergodic CP-distribution $Q^N$, let
  \begin{equation*}
    \mathbf{n} = (Nt_1,\ldots,Nt_k,Nt_1,\ldots,Nt_k,\ldots).
  \end{equation*}
  Write $\hat R = \times_{i=1}^k \overline{R}_i$ and define an adapted distribution $P^N$ by
  \begin{equation*}
    \overline{P} = \spl_{\mathbf{n}} \times_{i=1}^\infty \hat R.
  \end{equation*}
  Furthermore, define $Q^N$ by
  \begin{equation*}
    \overline{Q}^N = \frac{1}{N} \sum_{j=0}^{N-1} \overline{Q}_j,
  \end{equation*}
  where $\overline{Q}_j$ is the pushforward of $P^N$ under $(\mu,\iii) \mapsto \mu_{\iii|_j}$. It follows from the Birkhoff ergodic theorem that
  \begin{equation*}
    \lim_{L \to \infty} \frac{1}{L} \sum_{i=0}^{L-1} \delta_{\mu_{\iii|_{iN+j}}} = \overline{Q}_j
  \end{equation*}
  for $P$-almost all $(\mu,\iii)$. Therefore, $Q$ is a micromeasure distribution and, by Lemma \ref{thm:generated-CP}, a CP-distribution. Since a full $P^N$-measure set has positive $Q^N$-measure, the proof of Lemma \ref{thm:lemma18} implies that $Q^N$ is ergodic. To show that $Q^N \to \tfrac{1}{q} \sum_{i=1}^k t_iR_i$ as $N \to \infty$ is straightforward.
\end{proof}

\subsection{Uniformly scaling measures}
A measure $\mu \in \MM$ \emph{generates} a distribution $P^\square \in \PP(\MM^\square)$ at $x \in \R^2$ if
\begin{equation*}
  \TD(\mu,x) = \{ P^\square \}.
\end{equation*}
If $\mu$ generates $P^\square$ for $\mu$-almost all $x \in \R^2$, then $\mu$ is a \emph{uniformly scaling measure}. Note that, by Theorem \ref{thm:theorem11}, generated distributions are restricted fractal distributions and, by Theorem \ref{thm:EFD-generated}, restricted ergodic fractal distributions are generated by uniformly scaling measures. The following theorem shows that any restricted fractal distribution is generated by a uniform scaling measure.

\begin{theorem} \label{thm:theorem22}
  If $P^\square$ is a restricted fractal distribution on $\MM^\square$, then there exists a uniform scaling measure $\mu \in \MM$ generating $P^\square$.
\end{theorem}

\begin{proof}
  If $P^\square$ is a restricted ergodic fractal distribution, then, by Theorem \ref{thm:EFD-generated}, $P^\square$-almost every measure is an uniformly scaling measure generating $P^\square$. By Theorems \ref{thm:theorem20} and \ref{thm:theorem21}, it then suffices to show that the set of all restricted fractal distributions $P^\square$ for which there exists a uniformly scaling measure generating $P^\square$ is closed. Let $(P^\square_i)_{i = 1}^\infty$ be a sequence of restricted ergodic fractal distributions converging to $P^\square$. Our goal is to find a uniformly scaling measure $\mu$ generating $P^\square$.
  
  Observe that, by Theorem \ref{thm:FDs-are-CPs}, there exist a CP-distribution $Q$ and ergodic CP-distributions $Q_i$ such that $(\cent\hat Q)^\square = P^\square$, $(\cent\hat Q_i)^\square = P_i^\square$ for all $i \in \N$, and $Q_i \to Q$ as $i \to \infty$. It follows from the proof of Lemma \ref{thm:lemma18} that for each $i \in \N$ there exists $\mu_i$ generating $Q_i$ for $\mu_i$-almost all $x$. If we can now show that there exists a measure $\mu$ which generates $Q$ at $\mu$-almost all $x \in \R^2$, then the proof of Theorem \ref{thm:theorem19} guarantees that $\mu$ also generates $(\cent\hat Q)^\square$ at $\mu$-almost all $x \in \R^2$. This will finish the proof. Note that the only role of the random translation used in the proof is to guarantee that the micromeasure distribution is a CP-distribution. To obtain such a measure $\mu$, we splice the measures $\mu_i$ together.
  
  Fix $0<\eps<1$ and choose $\eps_i \downarrow 0$ such that
  \begin{equation*}
    \prod_{i=1}^\infty (1-\eps_i) = \eps.
  \end{equation*}
  Since $\mu_i$ generates $Q_i$, there exists $m_i \in \N$ such that $\mu_i(U_i) > 1-\eps_i$, where
  \begin{equation*}
    U_i = \{ x \in [-1,1]^2 : d(\overline{\la \mu_i,x \ra}_N, \overline{Q}_i) < \eps_i \text{ for all } N \ge m_i \}
  \end{equation*}
  and $d$ is a metric giving the weak topology. Let $n_1 = \max\{ e^{m_1}, e^{m_2} \}$ and
  \begin{equation*}
    n_i = m_i + \max\{ e^{n_{i-1}}, e^{m_i}, e^{m_{i+1}} \}
  \end{equation*}
  for all $i \ge 2$. Recalling the splicing operation $\spl_{\mathbf{n}}$ introduced in the proof of Theorem \ref{thm:theorem21}, we define
  \begin{equation*}
    \mu' = \spl_{\mathbf{n}}\times_{i=1}^\infty\mu_i \quad \text{and} \quad
    U = \spl_{\mathbf{n}}(\times_{i=1}^\infty U_i).
  \end{equation*}
  Since
  \begin{equation*}
    \mu'(U) = \lim_{M \to \infty} \prod_{i=1}^M \mu_i(U_i) \ge \lim_{M \to \infty} \prod_{i=1}^M (1-\eps_i) = \eps > 0,
  \end{equation*}
  the normalized measure $\mu=\mu_U'$ is well-defined. It follows from the Besicovitch density theorem (see e.g.\ \cite[Theorem 2.14(1)]{Mattila1995}) that
  \begin{equation*}
    \MD(\mu,x) = \MD(\mu',x)
  \end{equation*}
  for $\mu'$-almost all $x \in U$. Thus, by showing that $\overline{\la \mu',x \ra}_N \to \overline{Q}$ as $N \to \infty$ for $\mu'$-almost all $x \in U$, we have finished the proof since then $\mu$ is the uniformly scaling measure we are seeking for. The proof of this fact is technical and the interested reader can find it at \cite[Lemma 5.5]{KaenmakiSahlstenShmerkin2015a}.
\end{proof}

\section{Conical densities via scenery flow} \label{sec:conical-scenery}

With the dynamical machinery built around the scenery flow, we are now ready to study conical densities. The presentation in this section follows \cite{Hochman2010, KaenmakiSahlstenShmerkin2015b}.

In \S \ref{sec:conical-densities}, we studied conical density results by using the Euclidean norm but in \S \ref{sec:scenery-flow}, we used $L^\infty$-norm. There is nothing special about the $L^\infty$-norm besides geometrically easy correspondence between CP-distributions and fractal distributions. All the results concerning fractal and tangent distributions remain the same regardless of the choice of the norm. Changing the norm for the extended fractal distribution only changes the normalization. The restricted fractal distributions basically only require the one-to-one correspondence between the extended and restricted versions. Therefore, we continue working with the Euclidean norm.

\subsection{Fractal distributions and dimension}
We shall first study how fractal distributions are related to dimension. The \emph{lower local dimension} of $\mu \in \MM$ at $x \in \R^2$ is defined by
\begin{equation} \label{eq:lower-local}
  \ldimloc(\mu,x) = \liminf_{r \downarrow 0} \frac{\log \mu(B(x,r))}{\log r}
\end{equation}
and the \emph{lower Hausdorff dimension} of $\mu$ is
\begin{equation*}
  \ldimh(\mu) = \essinf_{x \sim \mu} \ldimloc(\mu,x).
\end{equation*}
A measure is \emph{exact-dimensional} if the limit in \eqref{eq:lower-local} exists and is $\mu$-almost everywhere constant. In this case, the limit in \eqref{eq:lower-local} is denoted by $\dimloc(\mu,x)$ and the almost sure common value by $\dim(\mu)$.

\begin{lemma}
  If $\mu \in \MM$, then
  \begin{equation*}
    \ldimh(\mu) = \inf\{ \dimh(A) : A \subset \R^2 \text{ is a Borel set with } \mu(A)>0 \}.
  \end{equation*}
\end{lemma}

\begin{proof}
  Let $s < \ldimh(\mu)$. This means that for $\mu$-almost every $x \in \R^2$ there exists $r_x>0$ such that $\mu(B(x,r)) < r^s$ for all $0<r<r_x$. Let $F_\delta = \{ x \in \R^2 : \mu(B(x,r)) < r^s \text{ for all } 0<r<\delta \}$. Fix a Borel set $A \subset \R^2$ such that $\mu(A)>0$ and choose $\delta>0$ so small that $\mu(A \cap F_\delta)>0$. Let $\{ U_i \}_i$ be a $\delta$-cover of $A \cap F_\delta$. For each $i$ fix $x_i \in U_i \cap A \cap F_\delta$. Then
  \begin{equation*}
    \mu(U_i) \le \mu(B(x_i,\diam(U_i))) < \diam(U_i)^s
  \end{equation*}
  and
  \begin{equation*}
    \mu(A \cap F_\delta) \le \sum_i \mu(U_i) \le \sum_i \diam(U_i)^s.
  \end{equation*}
  Therefore, $\mu(A \cap F_\delta) \le \HH^s_\delta(A \cap F_\delta) \le \HH^s(A)$. Since $\bigcup_{\delta > 0} A \cap F_\delta = A$, we get $0 < \mu(A) \le \HH^s(A)$ and $\dimh(A) \ge s$.
  
  To show the other direction, assume to the contrary that $\ldimh(\mu)$ is strictly smaller than the claimed infimum. In other words, there exist a bounded Borel set $C \subset \R^2$ with $0<\mu(C)<\infty$ and $s>0$ such that
  \begin{equation} \label{eq:dim-antithesis}
    \ldimloc(\mu,x) < s < \inf\{ \dimh(A) : A \subset \R^2 \text{ is a Borel set with } \mu(A)>0 \}
  \end{equation}
  for all $x \in C$. By the Besicovitch density theorem (see e.g.\ \cite[Theorem 2.14(1)]{Mattila1995}), for every $x \in C$ there exists a sequence $r_i \downarrow 0$ such that $\mu|_C(B(x,r_i)) > r_i^s/2$. Fix $\delta>0$ and let
  \begin{equation*}
    \CC = \{ B(x,r) : x \in C,\; 0<r\le\delta, \text{ and } \mu|_C(B(x,r)) > r^s/2 \}.
  \end{equation*}
  Note that $C \subset \bigcup_{B \in \CC} B$ regardless of the choice of $\delta$. By the $5r$-covering theorem (see e.g.\ \cite[Theorem 2.1]{Mattila1995}), there exists a countable collection of mutually disjoint balls $B_i \in \CC$ such that $\bigcup_{B \in \CC} B = \bigcup_i 5B_i$. Thus $\{ 5B_i \}_i$ is a $10\delta$-cover of $C$ and
  \begin{align*}
    \HH^s_{10\delta}(C) &\le \sum_i \diam(5B_i)^s = 10^s\sum_i \rad(B_i)^s \\
    &\le 2 \cdot 10^s \sum_i \mu|_C(B_i) \le 2 \cdot 10^s\mu(C) < \infty,
  \end{align*}
  where $\rad(B)$ is the radius of the ball $B$. It follows that $\dimh(C) \le s$, a contradiction with \eqref{eq:dim-antithesis}.
\end{proof}

\begin{lemma} \label{thm:FD-dim}
  If $P^\square$ is a restricted fractal distribution on $\MM^\square$, then $P^\square$-almost every measure is exact-dimensional. Furthermore, if $P^\square$ is ergodic, then the value of the dimension is $P^\square$-almost everywhere constant and given by
  \begin{equation*}
    \int_{\MM^\square} \frac{\log \mu(B(0,r))}{\log r} \dd P^\square(\mu)
  \end{equation*}
  for all $0<r<1$.
\end{lemma}

\begin{proof}
  Let $P$ be an extended ergodic fractal distribution and $0<r<1$. Set
  \begin{equation*}
    F(\mu) = \frac{\log\mu(B(0,r))}{\log r}
  \end{equation*}
  and notice that
  \begin{equation} \label{eq:MH-30}
  \begin{split}
    \frac{\log\mu(B(0,r^N))}{\log r^N} &= \frac{1}{\log r^N} \sum_{n=1}^N \log\frac{\mu(B(0,r^n))}{\mu(B(0,r^{n-1}))} \\
    &= \frac{1}{N} \sum_{n=1}^N \frac{\log S^*_{n\log r^{-1}}\mu(B(0,r))}{\log r} \\
    &= \frac{1}{N} \sum_{n=1}^N F((S^*_{\log r^{-1}})^n \mu).
  \end{split}
  \end{equation}
  By letting $N \to \infty$, the Birkhoff ergodic theorem implies that
  \begin{equation*}
    \dimloc(\mu,0) = \int_{\MM^*} F(\nu) \dd P(\nu)
  \end{equation*}
  for $P$-almost all $\mu \in \MM^*$. Finally, by the quasi-Palm property,
  \begin{equation*}
    \dimloc(\mu,x) = \int_{\MM^*} F(\nu) \dd P(\nu)
  \end{equation*}
  for $P$-almost all $\mu \in \MM^*$ and for $\mu$-almost all $x \in \R^2$.
\end{proof}

The \emph{dimension of a restricted fractal distribution} $P^\square$ is defined by
\begin{equation*}
  \dim(P^\square) = \int_{\MM^\square} \dim(\mu) \dd P^\square(\mu).
\end{equation*}
Note that, by the ergodic decomposition and Lemma \ref{thm:FD-dim},
\begin{equation*}
  \dim(P^\square) = \int_{\MM^\square} \frac{\log\mu(B(0,r))}{\log r} \dd P^\square(\mu)
\end{equation*}
for all $0<r<1$.

\begin{lemma} \label{thm:FD-dim-cont}
  The function $P^\square \mapsto \dim(P^\square)$ defined on the set of all restricted fractal distributions is continuous.
\end{lemma}

\begin{proof}
  A given measure can be a discontinuity point for the function $\mu \mapsto \mu(B(0,r))$ for at most countably many $0<r<1$. Hence, by the dominated convergence theorem, the function
  \begin{equation*}
    \hat F(\mu) = \int_0^{-\log r} F(S_t^*\mu) \dd t
  \end{equation*}
  is continuous at every $\mu \in \MM^\square$ with $0 \in \spt(\mu)$. By Fubini's theorem, $\dim(P^\square) = \int_{\MM^\square} \hat F(\mu) \dd P^\square(\mu)$.
\end{proof}

Recall that for a given measure $\mu \in \MM$, by Theorem \ref{thm:theorem11}, a tangent distribution $P \in \TD(\mu,x)$ is a restricted fractal distribution for $\mu$-almost all $x \in \R^2$.

\begin{theorem} \label{thm:FD-dimloc}
  If $\mu \in \MM$, then
  \begin{equation*}
    \ldimloc(\mu,x) = \inf\{ \dim(P) : P \in \TD(\mu,x) \}
  \end{equation*}
  for $\mu$-almost all $x \in \R^2$. In particular, if $\mu$ is a uniformly scaling measure generating $P$, then $\mu$ is exact-dimensional and $\dim(\mu)=\dim(P)$.
\end{theorem}

\begin{proof}
  Let $\mu \in \MM$ and choose $x \sim \mu$. By Theorem \ref{thm:theorem11}, every member of $\TD(\mu,x)$ is a restricted fractal distribution. Let $P \in \TD(\mu,x)$ and $\hat F$ be as in the proof of Lemma 4.3. Since $\hat F$ is continuous, we have
  \begin{equation*}
    \dim(P) = \int_{\MM^\square} \hat F(\mu) \dd P^\square(\mu) = \lim_{i \to \infty} \frac{1}{T_i} \int_0^{T_i} \hat F(\mu_{x,t}^\square) \dd t
  \end{equation*}
  for some sequence $(T_i)_{i \in \N}$. By choosing $r=1/e$, \eqref{eq:MH-30} implies that
  \begin{equation*}
    \frac{1}{N} \sum_{n=1}^N \hat F(\mu_{x,n}^\square) = \int_0^1 \frac{1}{N} \sum_{n=1}^N F(S_t^\square \mu_{x,n}) \dd t = \int_0^1 \frac{1}{N} \log\frac{\mu(B(x,e^{-N-t-1}))}{\mu(B(x,e^{-t-1}))} \dd t.
  \end{equation*}
  Note that the integrand above is bounded and its limit does not depend on $t$. Thus, by the dominated convergence theorem,
  \begin{equation*}
    \liminf_{N \to \infty} \int_0^1 \frac{1}{N} \log\frac{\mu(B(x,e^{-N-t-1}))}{\mu(B(x,e^{-t-1}))} \dd t = \ldimloc(\mu,x).
  \end{equation*}
  Since
  \begin{equation*}
    \frac{1}{N} \sum_{n=1}^N \hat F(\mu_{x,n}^\square) \quad \text{and} \quad
    \frac{1}{T} \int_0^T \hat F(\mu_{x,n}^\square) \dd t
  \end{equation*}
  have the same asymptotics, we have shown $\ldimloc(\mu,x) \le \dim(P)$ for all $P \in \TD(\mu,x)$. Choosing $P$ to be an accumulation point for the sequence that realizes $\ldimloc(\mu,x)$ finishes the proof.
\end{proof}

\begin{lemma} \label{lem:onlydim}
  If $P^\square$ is a restricted fractal distribution, then
  \begin{itemize}
    \item[(1)] $\dim(P^\square)=0$ if and only if $P^\square=\delta_{\delta_0}$,
    \item[(2)] $\dim(P^\square)=2$ if and only if $P^\square=\delta_{\LL}$,
  \end{itemize}
  where $\delta_x$ is the Dirac mass at $x$ and $\LL$ is the normalized Lebesgue measure.
\end{lemma}

\begin{proof}
  Since clearly $\dim(\delta_{\delta_0}) = 0$, let $P^\square$ be a restricted fractal distribution such that $\dim(P^\square) = 0$. By Lemma \ref{thm:FD-dim}, we have
  \begin{equation*}
    0 = \dim(P^\square) = \int_{\MM^\square} \frac{\log\mu(B(0,1/n))}{\log 1/n} \dd P^\square(\mu)
  \end{equation*}
  for all $n \in \N$. It follows that $\mu(B(0,1/n)) = 1$ for $P^\square$-almost all $\mu \in \MM^\square$ and all $n \in \N$. Therefore, $\mu = \delta_0$ for $P^\square$-almost all $\mu \in \MM^\square$, which implies that $P = \delta_{\delta_0}$.

  To show the second claim, observe first that $\dim(\delta_{\LL}) = 2$. Let $P^\square$ be a restricted fractal distribution such that $\dim(P)=2$. By Theorem \ref{thm:theorem10}, it suffices to prove the claim for an ergodic fractal distribution. Recalling Theorem \ref{thm:FDs-are-CPs}, let $Q$ be an ergodic CP-distribution with $(\cent \hat Q)^\square = P^\square$. Since, by Theorem \ref{thm:EFD-generated}, $P^\square$-almost every $\mu \in \MM^\square$ is a uniformly scaling measure and thus, by Theorem \ref{thm:FD-dimloc}, $\dim(\mu) = \dim(P) = 2$ for $P^\square$-almost every $\mu \in \MM^\square$. It follows that $\dim(\mu) = 2$ for $\overline{Q}$-almost all $\mu$. By \cite[Lemma 2.1]{Furstenberg2008}, we see that
  \begin{equation*}
    2 = \dim(\mu) = \frac{h(\mu)}{\log 2} = -\int_{\MM^\square \times [-1,1]} \frac{\log\nu(D_1(x))}{\log 2} \dd Q(\nu,x)
  \end{equation*}
  for $\overline{Q}$-almost all $\mu$, where $h(\mu)$ is the measure-theoretic entropy of $\mu$ with respect to the dyadic cubes. Since the maximal entropy of any measure on $[-1,1]^2$ with respect to dyadic cubes is $\log 4$, we see by iterating that $P^\square$-almost every $\mu$ gives equal mass to every dyadic cube at a given level. It follows that $P^\square$-almost every $\mu$ is $\LL$ and hence, $P = \delta_{\LL}$.
\end{proof}

\subsection{Fractal distributions and conical densities}
Before we start examining connections between the scenery flow, conical densities, and rectifiability, we modify the quasi-Palm property for our purposes.

\begin{lemma} \label{thm:quasi-palm-mod}
  If $P^\square$ is a restricted fractal distribution with $\dim(P^\square) > 0$ and $\AA \subset \MM^\square$ is a Borel set with $P^\square(\AA)=1$, then $P^\square$-almost every $\nu \in \AA$ and $\nu$-almost every $z \in \R^2$ and for every $t \ge 0$ there exists $t_z \ge t$ such that $B(z,e^{-t_z}) \subset B(0,1)$ and $\nu_{z,t_z} \in \AA$.
\end{lemma}

\begin{proof}
  By the quasi-Palm property, it suffices to show that for $P^\square$-almost every $\nu$ we have $\nu(\partial B)=0$ for all balls $B$. Suppose to the contrary that the set
  \begin{equation*}
    \AA = \{ \nu \in \MM^\square : \nu(\partial B) > 0 \text{ for some ball } B \}
  \end{equation*}
  has positive measure, $P^\square(\AA)>0$. By the ergodic decomposition, we may assume that $P^\square$ is ergodic. According to Theorem \ref{thm:EFD-generated}, $P^\square$-almost every $\nu \in \AA$ is a uniformly scaling measure generating $P^\square$. Recalling that for each $\nu \in \AA$ there is a ball $B$ so that $\nu(\partial B)>0$, it follows from the Besicovitch density theorem (see e.g.\ \cite[Theorem 2.14(1)]{Mattila1995}) that for $P$-almost every $\nu$ the normalized restriction $\nu_{\partial B}$ is a uniformly scaling measure for $P$. Each such $\nu_{\partial B}$ is supported on a $(d-1)$-dimensional sphere $\partial B$ and hence $P$-almost every $\nu \in \AA$ is supported on a $(d-1)$-dimensional plane. This is because tangent measures of measures supported on $\partial B$ are supported on a $(d-1)$-dimensional plane. Thus, in particular, $P^\square$-almost every measure $\nu \in \AA$ is supported on a $(d-1)$-dimensional plane. Since $\nu$ is supported on a $(d-1)$-plane $V_1$ and $\nu(\partial B_1)>0$, we have $\nu(\partial B_1\cap V_1)>0$, where the intersection $\partial B_1 \cap V_1$ is either $(d-2)$-dimensional or a single point. If the intersection is one point, then $P^\square = \delta_{\delta_0}$ which is a contradiction by Lemma \ref{lem:onlydim}.

  Now we continue inductively and show that $P^\square$ almost every $\nu$ gives positive measure for a $(d-3)$-dimensional set $\partial B \cap V_2$ where $V_2 \in G(d,d-2)$. Eventually, we are at dimension $1$ in which case, since the intersection of a line and $\partial B \cap V_{d-2}$ where $V_{d-2} \in G(d,2)$ is at most two points, the conclusion is that $P^\square = \delta_{\delta_0}$. This contradiction finishes the proof.
\end{proof}

The following lemma is needed to be able to work with the weak convergence.

\begin{lemma} \label{thm:lemma24}
  Let $0<\alpha\le 1$. The sets
  \begin{equation*}
    \AA_\eps^\alpha = \{ \nu \in \MM^\square : \nu(X(0,1,\theta,\alpha)) \le \eps \text{ for some } \theta \in S^1 \}
  \end{equation*}
  are nested and closed for all $\eps \ge 0$, and the set
  \begin{equation*}
    \AA_0^\alpha = \{ \nu \in \MM^\square : \spt(\nu) \cap X(0,1,\theta,\alpha) = \emptyset \text{ for some } \theta \in S^1 \}
  \end{equation*}
  satisfies $\AA_0^\alpha \subset S_t^{-1}(\AA_0^\alpha)$ for all $t \ge 0$.
\end{lemma}

\begin{proof}
  That the sets $\AA^\alpha_\eps$ are nested follows immediately from the definition. Let $\nu_i \in \AA^\alpha_\eps$ and $\nu \in \MM^\square$ be such that $\nu_i \to \nu$. Let $\theta_i \in S^1$ be so that
  \begin{equation*}
    \nu_i(X(0,1,\theta_i,\alpha)) \le \eps.
  \end{equation*}
  By the compactness of $S^1$, there is $\theta \in S^1$ such that $\theta_i \to \theta$ along a subsequence. Now for each $0<\eta<1$ we have $i_\eta$ such that
  \begin{equation*}
    X(0,1,\theta,\eta\alpha) \subset X(0,1,\theta_i,\alpha)
  \end{equation*}
  for all $i \ge i_\eta$. Since the cones are open, we have
  \begin{equation*}
    \nu(X(0,1,\theta,\eta\alpha)) \le \liminf_{i \to \infty} \nu_i(X(0,1,\theta,\eta\alpha)) \le \eps
  \end{equation*}
  for all $0<\eta<1$. Letting $\eta \uparrow 1$ we get
  \begin{equation*}
    \nu(X(0,1,\theta,\alpha)) \le \eps
  \end{equation*}
  and $\nu \in \AA^\alpha_\eps$. The second claim follows directly from the definitions.
%
\end{proof}

The proof of the following result is based on showing that there cannot be ``too many'' rectifiable tangent measures. This means that, perhaps surprisingly, most of the known conical density results are, in some sense, manifestations of rectifiability. In particular, the results follow from Lemma \ref{thm:lemma1} by applying the machinery of fractal distributions. 

\begin{theorem} \label{thm:theorem25}
  If $1<s\le 2$ and $0 < \alpha \le 1$, then there exists $\eps = \eps(s,\alpha) > 0$ satisfying the following: For every Radon measure $\mu$ on $\R^2$ with $\ldimh(\mu) \ge s$ it holds that
  \begin{equation*}
    \liminf_{T \to \infty} \la \mu \ra_{x,T}(\MM^\square \setminus \AA^\alpha_\eps) \ge s-1
  \end{equation*}
  for $\mu$-almost all $x \in \R^2$.
\end{theorem}

\begin{proof}
  Let $p,\delta>0$ be such that
  \begin{equation*}
    p < s-1-\delta < s-1.
  \end{equation*}
  Suppose to the contrary that for each $\eps>0$ there exists a Radon measure $\mu$ with $\ldimh(\mu) \ge s$ such that
  \begin{equation*}
    \limsup_{T \to \infty}\la \mu \ra_{x,T}(\AA^\alpha_\eps) > 1-p
  \end{equation*}
  on a set $E_\eps$ of positive measure. By Theorems \ref{thm:theorem11} and \ref{thm:FD-dimloc}, we may assume that at points $x \in E_\eps$, all tangent distributions of $\mu$ are restricted fractal distributions and
  \begin{equation*}
    \inf\{ \dim(P) : P \in \TD(\mu,x) \} = \ldimloc(\mu,x) > s-\delta.
  \end{equation*}
  Fix $x \in E_\eps$. For each $\eps>0$, as $\AA^\alpha_\eps$ is closed by Lemma \ref{thm:lemma24}, we find a tangent distribution $P_\eps$ so that $P_\eps(\AA^\alpha_\eps) \ge 1-p$. Let $P$ be a weak limit of $P_\eps$ as $\eps \downarrow 0$. By Theorem \ref{thm:theorem20}, $P$ is a restricted fractal distribution, and by Lemma \ref{thm:FD-dim-cont}, we have
  \begin{equation*}
    \dim(P) \ge s-\delta.
  \end{equation*}
  Since the sets $\AA^\alpha_\eps$ are nested and closed, we have
  \begin{equation*}
    P(\AA^\alpha_\eps) \ge \limsup_{\eta \downarrow 0}P_\eta(\AA^\alpha_\eps) \ge \limsup_{\eta \downarrow 0}P_\eta(\AA^\alpha_\eta) \ge 1-p,
  \end{equation*}
  and thus,
  \begin{equation*}
    P(\AA^\alpha_0) = \lim_{\eps \downarrow 0}P(\AA^\alpha_\eps) \ge 1-p.
  \end{equation*}
  Theorem \ref{thm:theorem10} guarantees that $P$-almost every ergodic component $P_\omega$ is a restricted fractal distribution. By Lemma \ref{thm:lemma24}, we have $\AA^\alpha_0 \subset S_t^{-1}(\AA^\alpha_0)$ for all $t \ge 0$. Since $P_\omega$ is $S$-invariant, we have $P(\AA^\alpha_0) = P(S_t^{-1}(\AA^\alpha_0))$ for all $t \ge 0$. Therefore, $\AA^\alpha_0 = S_t^{-1}(\AA^\alpha_0)$ up to $P_\omega$-measure zero. Hence we have $P_\omega(\AA^\alpha_0) \in \{ 0,1 \}$ for $P$-almost all $\omega$.
  
  If $P_\omega(\AA^\alpha_0) = 0$, then we use the trivial estimate $\dim(P_\omega) \le 2$. Let us consider the case where $P_\omega(\AA^\alpha_0) = 1$. Since $P_\omega$ satisfies the quasi-Palm property of Lemma \ref{thm:quasi-palm-mod}, for $P_\omega$-almost every $\nu$ and $\nu$-almost every $z$ we have $\nu_{z,t_z} \in \AA^\alpha_0$ for some $t_z>0$ with $B(z,e^{-t_z}) \subset B(0,1)$. Thus if
  \begin{equation*}
    E_\nu = \{ z \in B(0,1) : \nu_{z,t_z} \in \AA^\alpha_0 \},
  \end{equation*}
  then for $P_\omega$-almost every $\nu$ we have $\nu(E_\nu)=1$. By the definition of $\AA^\alpha_0$, for every $z \in E$ there is $\theta \in S^1$ with
  \begin{equation*}
    E_\nu \cap X(z,e^{-t_z},\theta,\alpha) = \emptyset.
  \end{equation*}
  Lemma \ref{thm:lemma1} implies that $E_\nu$ is rectifiable. In particular, $\dim(\nu) \le 1$ and $\dim(P) \le 1$.
  
  Since $P_\omega(\AA^\alpha_0) \in \{ 0,1 \}$ for $P_\omega$-almost every $\omega$, we have
  \begin{equation*}
    1-p \le P(\AA^\alpha_0) = \int P_\omega(\AA^\alpha_0) \dd P(\omega) = P(\{ \omega : P_\omega(\AA^\alpha_0) = 1 \}).
  \end{equation*}
  Using this, we estimate
  \begin{align*}
    s-\delta &\le \dim(P) = \int \dim(P_\omega) \dd P(\omega) \\
    &\le P(\AA^\alpha_0) + 2(1-P(\AA^\alpha_0)) \le (1-p) + 2p.
  \end{align*}
  It follows that $p \ge s-1-\delta$ which is a contradiction with the choice of $\delta$.
\end{proof}

Theorem \ref{thm:theorem25} claims that there exists $\eps$ such that if $\ldimh(\mu) \ge s > 1$, then for the proportion $s-1$ of scales $e^t>0$ we have
\begin{equation*}
  \inf_{\theta \in S^1} \frac{\mu(X(x,e^{-t},\theta,\alpha))}{\mu(B(x,e^{-t}))} > \eps
\end{equation*}
for $\mu$-almost all $x \in \R^2$. In these scales, the measure $\mu$ is well distributed. The following theorem demonstrates that the proportion $s-1$ is sharp. Note that, by the rotation invariance of the Lebesgue measure $\LL^2$, the quantity
\begin{equation*}
  \frac{\LL^2(X(0,1,\theta,\alpha))}{\LL^2(B(0,1))}.
\end{equation*}
does not depend on $\theta \in S^1$. Let us denote it by $\eps(\alpha)$.

\begin{theorem} \label{thm:theorem26}
  If $1<s\le 2$ and $0 < \alpha \le 1$, then there exists a Radon measure $\mu$ on $\R^2$ with $\dim(\mu)=s$ such that
  \begin{equation*}
    \lim_{T \to \infty} \la \mu \ra_{x,T}(\MM^\square \setminus \AA^\alpha_\eps) =
    \begin{cases}
      s-1, &\text{if } 0<\eps<\eps(\alpha), \\
      0,   &\text{if } \eps>\eps(\alpha)
    \end{cases}
  \end{equation*}
  for $\mu$-almost all $x \in \R^2$.
\end{theorem}

\begin{proof}
  The measure $\mu$ is a uniformly scaling measure generating
  \begin{equation*}
    P^\square = (s-1)\delta_{\LL} + (1-(s-1))\delta_{\HH},
  \end{equation*}
  where $\LL$ is the normalization of $\LL^2|_{B(0,1)}$ and $\HH$ is the normalization of $\HH^1|_{\ell \cap B(0,1)}$ for a fixed line $\ell$ through the origin. Note that since $P^\square$ is a convex combination of two restricted fractal distributions, it is a restricted fractal distribution. Theorem \ref{thm:theorem22} quarantees the existence of such a measure $\mu$.
  
  By Theorem \ref{thm:FD-dimloc}, $\mu$ is exact-dimensional and
  \begin{equation*}
    \dim(\mu) = \dim(P^\square) = 2(s-1) + (1-(s-1)) = s.
  \end{equation*}
  Fix $0<\eps<\eps(\alpha)$. Since $\LL(X(0,1,\theta,\alpha)) = \eps(\alpha) > \eps$ for all $\theta \in S^1$ and $\HH(X(0,1,\ell^\bot,\alpha)) = 0$ we have
  \begin{equation*}
    P^\square(\MM^\square \setminus \AA^\alpha_\eps) = s-1.
  \end{equation*}
  Thus, by the weak convergence, it follows that
  \begin{equation*}
    \lim_{T \to \infty} \la \mu \ra_{x,T}(\MM^\square \setminus \AA^\alpha_\eps) = s-1.
  \end{equation*}
  The case $\eps > \eps(\alpha)$ is treated similarly.
\end{proof}

In Example \ref{ex:example9}, we found a purely unrectifiable measure $\mu$ and $\theta \in S^1$ such that for every $0 < \alpha < 1$
\begin{equation*}
  \lim_{r \downarrow 0} \frac{\mu(X(x,r,\theta,\alpha))}{\mu(B(x,r))} = 0
\end{equation*}
for $\mu$-almost all $x \in \R^2$. This reflects the fact that rectifiability can be broken by having the measure ``look unrectifiable'' at some sparse sequence of scales. In fact, $\mu$ is rectifiable if and only if at $\mu$-almost every point each tangent measure is $\HH^1|_{\ell \cap B(0,1)}$ for some line $\ell$; see \cite[Theorem 16.5]{Mattila1995}. Therefore, if $\mu$ is rectifiable, then at $\mu$-almost every $x$ any $P \in \TD(\mu,x)$ satisfies
\begin{equation*}
  P(\{ \nu \in \MM^\square : \spt(\nu) \text{ is rectifiable} \}) = 1.
\end{equation*}
Note, however, that there exists a purely unrectifiable measure $\mu$ such that
\begin{equation*}
  P(\{ \nu \in \MM^\square : \spt(\nu) \text{ is not rectifiable} \}) = 0
\end{equation*}
for all $P \in \TD(\mu,x)$ at $\mu$-almost all $x \in \R^2$; see \cite[\S 20]{DavidSemmes1991}. Motivated by these observations, we say that a Radon measure $\mu$ is \emph{$p$-average unrectifiable} for $0\le p<1$ if
\begin{equation*}
  P(\{ \nu \in \MM^\square : \spt(\nu) \text{ is not rectifiable} \}) > p
\end{equation*}
for every $P \in \TD(\mu,x)$ at $\mu$-almost all $x \in \R^2$. For example, if $\mu$ is the $1$-dimensional Hausdorff measure on $\R^2$ restricted to the product of two $\tfrac14$-Cantor sets, then
\begin{equation*}
  P(\{ \nu \in \MM^\square : \spt(\nu) \text{ is not rectifiable} \}) = 1
\end{equation*}
for every $P \in \TD(\mu,x)$ at $\mu$-almost all $x \in \R^2$; see e.g.\ \cite{Bandt2001}. On the other hand, any measure $\mu$ supported on a self-similar set $E$ satisfying the strong separation condition of dimension strictly less than $1$ fails to be $0$-average unrectifiable. Indeed, it follows from self-similarity that for any $x\in E$ and $P\in\TD(\mu,x)$, the support of $P$-almost every measure $\nu$ is contained in a homothetic copy of $E$. Recall that any set of upper Minkowski dimension strictly less than $1$ can be covered by a single Lipschitz curve; see e.g.\ \cite[Lemma 3.1]{BalkaHarangi2014}.

\begin{lemma} \label{lem:dim-unrectifiability}
  If $\mu$ is a Radon measure on $\R^2$ such that $\ldimh(\mu) > s > 1$, then
  \begin{equation*}
    P(\{ \nu \in \MM^\square : \spt(\nu) \text{ is not rectifiable} \}) > s-1
  \end{equation*}
  for every $P \in \TD(\mu,x)$ at $\mu$-almost all $x \in \R^2$. In particular, $\mu$ is $(s-1)$-average unrectifiable.
\end{lemma}

\begin{proof}
  By Theorems \ref{thm:theorem11} and \ref{thm:FD-dimloc}, at $\mu$-almost every $x \in \R^2$, all elements of $\TD(\mu,x)$ are restricted fractal distributions and
  \begin{equation*}
    s < \ldimloc(\mu,x) = \inf\{ \dim P : P \in \TD(\mu,x) \}.
  \end{equation*}
  Pick such a point $x$, choose any $P \in \TD(\mu,x)$, and write
  \begin{equation*}
    \AA = \{ \nu \in \MM^\square : \spt(\nu) \text{ is not rectifiable} \}.
  \end{equation*}
  Recall that if a measure $\nu$ has a rectifiable support, then $\ldimh(\nu) \le \dimh(\spt(\nu)) \le 1$. Moreover, any measure $\nu$ on $\R^2$ satisfies $\ldimh(\nu) \le 2$. Thus we deduce that
  \begin{equation*}
    s < \dim(P) = \int_{\MM^\square} \dim(\nu) \dd P(\nu) \le 1-P(\AA) + 2P(\AA) = 1 + P(\AA).
  \end{equation*}
  Hence $P(\AA) > s-1$, showing that $\mu$ is $(s-1)$-average unrectifiable as claimed.
\end{proof}

\begin{lemma} \label{lma:constructionofunrect}
  If $0 \leq p \leq 1$, then there exists a uniformly scaling measure $\mu$ generating a fractal distribution $P$ with
  \begin{equation*}
    P(\{\nu \in \MM^\square : \spt\nu \text{ is not strongly $k$-rectifiable}\}) = p.
  \end{equation*}
\end{lemma}

\begin{proof}
  The proof is similar to that of Theorem \ref{thm:theorem26}. Let $E \subset \R^2$ be a $1$-dimensional self-similar set satisfying a strong separation condition obtained from a finite number of homotheties. By the self-similarity of $E$, the Hausdorff measure $\HH^1|_E$ is a uniformly scaling measure generating a restricted ergodic fractal distribution $Q$ supported on measures $\nu$ such that $\spt(\nu)$ is a translated and scaled copy of $E$ restricted to the unit ball; see Bandt \cite{Bandt2001}. Thus for $Q$-almost every $\nu$ the support $\spt(\nu)$ is purely unrectifiable. This shows that
  \begin{equation*}
    Q(\{\nu \in \MM^\square : \spt(\nu) \text{ is not rectifiable}\}) = 1.
  \end{equation*}
  Now defining
  \begin{equation*}
    P = pQ + (1-p)\delta_{\HH},
  \end{equation*}
  where $\HH$ is again the normalization of $\HH^1|_{\ell \cap B(0,1)}$ for a fixed line $\ell$ through the origin, we obtain a restricted fractal distribution $P$ that satisfies
  \begin{equation*}
    P(\{\nu \in \MM^\square : \spt(\nu) \text{ is not rectifiable}\}) = p.
  \end{equation*}
  Theorem \ref{thm:theorem22} guarantees the existence of a uniformly scaling measure $\mu$ generating $P$, so the proof is complete.
\end{proof}

The following theorem is the main conical density result for average unrectifiable measures. By Lemma \ref{lem:dim-unrectifiability}, the result generalizes Theorem \ref{thm:theorem25} for a fixed measure. Observe that, in Theorem \ref{thm:theorem25}, the value of $\eps$ is independent of $\mu$ whereas here it is not. We say that a Radon measure is \emph{regular} if
\begin{equation*}
  C^{-1}r \le \mu(B(x,r)) \le Cr
\end{equation*}
for all $r>0$ and $\mu$-almost all $x \in \R^2$.

\begin{theorem} \label{thm:theorem27}
  If $0 \le p<1$ and $\mu$ is a $p$-average unrectifiable measure, then for every $0<\alpha\le 1$ there exists $0<\eps<1$ such that
  \begin{equation} \label{eq:con-dens-result}
    \liminf_{T \to \infty} \la \mu \ra_{x,T}(\MM^\square \setminus \AA^\alpha_\eps) > p
  \end{equation}
  for $\mu$-almost all $x \in \R^2$. Conversely, if a regular measure $\mu$ satisfies \eqref{eq:con-dens-result} for some $0 \le p < 1$ and $0<\alpha,\eps<1$, then $\mu$ is $p$-average unrectifiable.
\end{theorem}

\begin{proof}
  To show the first claim, we observe that, by the Besicovitch density theorem (see e.g.\ \cite[Corollary 2.14(1)]{Mattila1995}), it suffices to prove the statement on a set of positive $\mu$-measure. Suppose to the contrary that a $p$-average unrectifiable measure $\mu$ and $0<\alpha\le 1$ are such that for each $0<\eps\le 1$ the condition \eqref{eq:con-dens-result} fails to hold in a set $E_\eps$ of full $\mu$-measure. Recalling Theorem \ref{thm:theorem11}, we may assume that all tangent distributions at points $x \in E_\eps$ are restricted fractal distributions and satisfy $P(\{ \nu \in \MM^\square : \spt(\nu) \text{ is rectifiable} \}) < 1-p$. Let $x \in \bigcap E_\eps$, where the intersection is over all rational $0<\eps\le 1$. Then
  \begin{equation*}
    \limsup_{T \to \infty} \la \mu \ra_{x,T}(\AA^\alpha_\eps) > 1-p
  \end{equation*}
  for all rational $0<\eps\le 1$. We choose a tangent distribution $P$ so that $P(\AA^\alpha_\eps) \ge 1-p$ for all rational $0<\eps\le 1$. Since, by Lemma \ref{thm:lemma24}, the sets $\AA^\alpha_\eps$ are nested and closed, we get $P(\AA^\alpha_0) \ge 1-p$.

  Recall that, by Theorem \ref{thm:theorem10}, $P$-almost every ergodic component $P_\omega$ is a restricted fractal distribution. Furthermore, as remarked in the proof of Theorem \ref{thm:theorem25}, the set $\AA^\alpha_0$ is $S_t$-invariant up to $P$-measure zero. Thus we have $P_\omega(\AA^\alpha_0) \in \{0,1\}$ for $P$-almost every ergodic component $P_\omega$. Notice that
  \[
    P(\{ \omega : P_\omega(\AA^\alpha_0) = 1 \}) = P(\AA^\alpha_0) \ge 1-p.
  \]
  If $P_\omega(\AA^\alpha_0) = 1$, then, by the quasi-Palm property of Lemma \ref{thm:quasi-palm-mod}, for $P_\omega$-almost every $\nu$ and for $\nu$-almost every $z$ the normalized translation $\nu_{z,t_z}$ is an element of $\AA^\alpha_0$ for some $t_z > 0$ with $B(z,e^{-t_z}) \subset B(0,1)$. For each such $\nu$ let $E = \{ z \in B(0,1) : \nu_{z,t_z} \in \AA^\alpha_0 \}$ be this set of full $\nu$-measure. Thus for every $z \in E$ there are $V \in G(d,d-k)$ and $\theta \in S^{d-1}$ with
  $$
    E \cap X(z,e^{-t_z},V,\alpha) \setminus H(z,\theta,\alpha) = \emptyset.
  $$
  Lemma \ref{thm:mattila} implies that $E$ is rectifiable. Consequently,
  \[
    P_\omega(\{ \nu \in \MM^\square : \spt\nu \text{ is rectifiable} \}) = 1.
  \]
  Thus by the ergodic decomposition
  $$
    P(\{ \nu \in \MM^\square : \spt\nu \text{ is strongly $k$-rectifiable} \}) \geq 1-p.
  $$
  The proof of the first claim is now finished since this contradicts the $p$-average unrectifiability assumption.
  
  To show the second claim, let $0 \le p < 1$ and $0<\alpha,\eps<1$ be such that a regular measure $\mu$ satisfies \eqref{eq:con-dens-result} for $\mu$-almost all $x \in \R^2$ with these parameters. Fix a point $x$ such that \eqref{eq:con-dens-result} is satisfied at $x$ and all tangent distributions at $x$ are restricted fractal distributions; recall Theorem \ref{thm:theorem11}. Let
  \begin{equation} \label{eq:large-mass-in-cones}
    \BB =  \{ \nu \in \MM^\square : \nu(\overline{X(0,1,\theta,\alpha)}) \ge \eps \text{ for all $\theta \in S^1$} \}.
  \end{equation}
  As in Lemma \ref{thm:lemma24}, we see that the set $\BB$ is closed. Using this and the hypothesis \eqref{eq:con-dens-result}, we have
  \begin{equation*}
    P(\BB) \ge \liminf_{T \to \infty} \la \mu \ra_{x,T}(\BB) > p
  \end{equation*}
  for all $P \in \TD(\mu,x)$.
  
  Let $P$ be a restricted ergodic fractal distribution such that
  \[
    P(\{\nu\in\MM^\square : \nu \text{ is rectifiable}  \}) > 0.
  \]
  By Theorem \ref{thm:theorem22}, there exists a uniformly scaling measure which is rectifiable and generates $P$. Then it follows from \cite[Theorem 16.5]{Mattila1995} that there exists a line $\ell$ through the origin such that $P=\delta_{\HH}$, where $\HH$ is the normalized restriction of $\HH^1$ to $\ell \cap B(0,1)$.
  
  Since $\mu$ was assumed to be regular, \cite[Lemma 14.7(1)]{Mattila1995} implies that, for $\mu$-almost every $x$, all tangent measures $\nu$ at $x$ are regular and hence $\nu \ll \HH^1$. Therefore, if $P_\omega$ is an ergodic component of $P \in \TD(\mu,x)$ satisfying
  \[
    P_\omega(\{\nu\in\MM^\square : \spt(\nu) \text{ is rectifiable}\}) > 0,
  \]
  then there exists a line $\ell$ through the origin such that $P=\delta_{\HH}$, where $\HH$ is the normalized restriction of $\HH^1$ to $\ell \cap B(0,1)$, and hence $P_\omega(\BB)=0$. Thus
  \[
    p <  P(\BB) = \int P_\omega(\BB) \dd P(\omega) \le P(\{\nu\in\MM^\square : \spt(\nu) \text{ is not rectifiable}\}).
  \]
  As $x$ was a $\mu$-typical point and $P\in\TD(\mu,x)$ was arbitrary, this finishes the proof of the second claim.
\end{proof}


\end{document}